\documentclass[11pt]{amsart}
\usepackage{amsmath,amsfonts,amssymb,amsthm,a4wide,amsbsy,vaucanson-g}
\usepackage[utf8]{inputenc}
\usepackage{enumerate}
\usepackage{color,soul}

\usepackage{graphicx}

\DeclareMathOperator{\rep}{rep}
\DeclareMathOperator{\lcm}{lcm}

\theoremstyle{plain}
\newtheorem{theorem}{Theorem}
\newtheorem{corollary}[theorem]{Corollary}
\newtheorem{prop}[theorem]{Proposition}

\newtheorem{lemma}[theorem]{Lemma}

\theoremstyle{definition}
\newtheorem{definition}[theorem]{Definition}
\newtheorem{remark}[theorem]{Remark}
\newtheorem{example}[theorem]{Example}

\newtheorem{algo}{Algorithm}

\def\bw{\mathbf{w}}
\def\bu{\mathbf{u}}
\def\bv{\mathbf{v}}
\def\b1{\mathbf{1}}

\def\N{\mathbb{N}}
\def\R{\mathbb{R}}

\def\p{\mathsf{p}}
\def\dil{\mathrm{Dil}}
\def\spec{\mathrm{Spec}}
\def\multspec{\mathrm{Spec}}

\def\restriction#1#2{\mathchoice
              {\setbox1\hbox{${\displaystyle #1}_{\scriptstyle #2}$}
              \restrictionaux{#1}{#2}}
              {\setbox1\hbox{${\textstyle #1}_{\scriptstyle #2}$}
              \restrictionaux{#1}{#2}}
              {\setbox1\hbox{${\scriptstyle #1}_{\scriptscriptstyle #2}$}
              \restrictionaux{#1}{#2}}
              {\setbox1\hbox{${\scriptscriptstyle #1}_{\scriptscriptstyle #2}$}
              \restrictionaux{#1}{#2}}}
\def\restrictionaux#1#2{{#1\,\smash{\vrule height .8\ht1 depth .85\dp1}}_{\,#2}}

\begin{document}
\title[Asymptotic properties of morphisms]{Asymptotic properties of free monoid morphisms}

\author{\'Emilie CHARLIER}
\author{Julien LEROY}
\author{Michel RIGO}
\address{Universit\'e de Li\`ege, Institut de math\'ematique, Grande traverse 12 (B37),
4000 Li\`ege, Belgium\newline
echarlier@ulg.ac.be, J. Leroy@ulg.ac.be, M.Rigo@ulg.ac.be}

\begin{abstract} 
    Motivated by applications in the theory of numeration systems and
    recognizable sets of integers, this paper deals with morphic words
    when erasing morphisms are taken into account. 
    Cobham showed that if an infinite word $\bw =g(f^\omega(a))$  is the image of
    a fixed point of a morphism $f$ under another morphism $g$, then there exist 
    a non-erasing morphism $\sigma$ and a coding $\tau$ such that $\bw =\tau(\sigma^\omega(b))$.

Based on the Perron theorem about asymptotic properties of
powers of non-negative matrices, our main contribution is an in-depth study of the
growth type of iterated morphisms when one replaces erasing morphisms
with non-erasing ones. 
We also explicitly provide an algorithm computing $\sigma$ and $\tau$ from $f$ and $g$.
\end{abstract}


\maketitle

\section{Introduction}

Infinite words, i.e., infinite sequences of symbols from a finite set usually called alphabet, 
form a classical object of study. 
They have an important representation power: they are a natural way to code elements of an
infinite set using finitely many symbols, e.g., the coding of an orbit in a discrete dynamical system or 
the characteristic sequence of a set of integers. 
A rich family of infinite words, with a simple algorithmic description, 
is made of the words obtained by iterating a morphism \cite{Choffrut&Karhumaki:1997}. 
The necessary background about words is given in Section~\ref{sec:basic}.

In relation with numeration systems, recognizable sets of integers are well studied. 
For instance, see \cite{BHMV:1994}. Let $k\ge 2$ be an integer. 
A set $X\subseteq\N$ is said to be {\em $k$-recognizable} if the set of base-$k$ expansions 
of the elements in $X$ is accepted by a finite automaton. 
Characteristic sequences of $k$-recognizable sets have been characterized by Cobham \cite{Cobham:1972}. 
They are the images of a fixed point of a $k$-uniform morphism under a coding (also called letter-to-letter morphism). 
We let $A^*$ denote the set of finite words over the alphabet $A$. 
This set, equipped with a product which is the usual concatenation of words, is a monoid. 
A {\em morphism} $f\colon A^*\to B^*$ satisfies, for all $u,v\in A^*$, $f(uv)=f(u)f(v)$. 
A morphism is {\em $k$-uniform} if the image of every letter is a word of length $k$. 
A $1$-uniform morphism is a {\em coding}. 
As an example of recognizable set, the Baum--Sweet set $S$ is defined as follows \cite{Allouche:1987}. 
The integer $n$ belongs to $S$ if and only if the base-$2$ expansion of $n$  
contains no block of consecutive $0$’s of odd length. 
The set $S$ is $2$-recognizable, the deterministic automaton depicted in Figure~\ref{fig:baum} 
recognizes the base-$2$ expansions of the elements in $S$ (read most significant digit first).
 \begin{figure}[htbp]
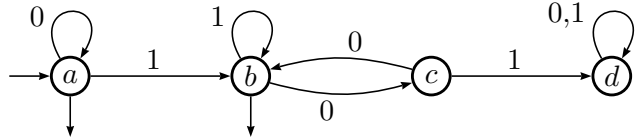

        \centering
\VCDraw{
        \begin{VCPicture}{(0,-1.3)(12,1.8)}
\State[a]{(0,0)}{1} 
\State[b]{(4,0)}{2}
\State[c]{(8,0)}{3}
\State[d]{(12,0)}{4}
\Initial[w]{1}
\Final[s]{1}
\Final[s]{2}
\LoopN{1}{$0$}
\LoopN{2}{$1$}
\LoopN{4}{$0,1$}
\ArcR{2}{3}{$0$}
\ArcR{3}{2}{$0$}
\EdgeL{1}{2}{$1$}
\EdgeL{3}{4}{$1$}
\end{VCPicture}
}
        \caption{The Baum--Sweet set is $2$-recognizable.}
        \label{fig:baum}
 \end{figure} 
The characteristic sequence $\mathbf{x}$ of $S$ starts with $1101100101001001\cdots$. 
It is the image of the infinite word $ab cb bdcb cbddbdcb\cdots$ 
under the coding $\tau\colon a,b\mapsto 1, c,d\mapsto 0$. 
Moreover, the latter infinite word is a fixed point of the $2$-uniform morphism 
$\sigma\colon a\mapsto ab, b\mapsto cb, c\mapsto bd, d\mapsto dd$. 
We write $\mathbf{x}=\tau(\sigma^\omega(a))$. 
Indeed, to obtain $\mathbf{x}$, one iterates the morphism $\sigma$ from $a$ 
to get a sequence $(\sigma^n(a))_{n\ge 0}$ of finite words of increasing length 
whose first terms are: $a, ab, abcb,  ab cb bdcb, ab cb bdcb cbddbdcb,\ldots$. 
This sequence converges to an infinite word which is a fixed point of $\sigma$. 
See, for instance, \cite{cant:2010,rigo1} for the definition of converging sequences of words. 
Note that there are infinitely many morphisms that can be used to generate the word $\mathbf{x}$. 
Take $\sigma'\colon a\mapsto abe, b\mapsto cefb, c\mapsto bfd, d\mapsto defd, e\mapsto ef, 
f\mapsto \varepsilon$ where $\varepsilon$ is the empty word (the identity element for concatenation), i.e., 
the unique word of length $0$. 
In that case, we say that $\sigma'$ is {\em erasing}. 
Take $\tau'\colon a,b\mapsto 1, c,d\mapsto 0, e,f\mapsto \varepsilon$. 
One fixed point of $\sigma'$ starts with $abe cefbef bfd efcefb\cdots$ 
and the image by the erasing morphism $\tau'$ of this word again is $\mathbf{x}$. 
The general aim of this paper is to derive from erasing morphisms such as $\sigma'$ 
and  $\tau'$ new non-erasing morphisms (where images of all letters have positive length) such as $\sigma$ and  $\tau$ 
that produce the same infinite word $\mathbf{x}$ 
and to retrieve some kind of canonical information 
(e.g., spectral radius, growth order) about $\mathbf{x}$ itself.

In the theory of integer base systems, we also recall another important theorem of Cobham \cite{Cobham:1969}. 
Let $k,\ell\ge 2$ be two multiplicatively independent integers, i.e., 
they are such that $\log k/\log \ell$ is irrational. 
If a set $X\subseteq\N$ is both $k$-recognizable and $\ell$-recognizable, 
then $X$ is a finite union of arithmetic progressions.  
In terms of morphisms, this result can be stated as follows. 
If an infinite word can be obtained as the coding of fixed points of two morphisms, 
one being $k$-uniform and the other one being $\ell$-uniform, 
then this word is ultimately periodic. It is of the form $uv^\omega=uvvv\cdots$, i.e., 
it has a (possibly empty) prefix $u$ followed by an infinite repetition of the finite (non-empty) word $v$. 

{\em Abstract numeration system} generalize in a natural way base-$k$ numeration systems, 
as well as many other classical systems such as the Zeckendorf system based on the Fibonacci sequence. 
Recognizability of sets of integers within an abstract numeration system has been fruitfully introduced. 
For a survey on these topics, see \cite[Chap. 3]{cant:2010}. 
Briefly, a set $X\subseteq\N$ is recognizable if the set $\rep(X)$ of the representations of its elements
within the considered numeration system  is a regular language.
In particular, the theorem of Cobham from 1972 can be extended as follows \cite{Rigo:2000,Maes&Rigo:02}. 
A set $X\subseteq\N$ is recognizable within one abstract numeration system (based on a regular language) 
if and only if its characteristic sequence $\chi_X$ is {\em morphic}: 
it is the image of a fixed point of a morphism under a morphism. 
In comparison with Cobham's result, there is no restriction on the two morphisms. 
In particular, the constructive proof in \cite{Maes&Rigo:02} yields morphisms that are usually erasing.

Since abstract numeration systems are a generalization of integer base systems, 
it is natural to seek an analogue of the theorem of Cobham from 1969. 
We state the corresponding results in terms of infinite words of the form $g(f^\omega(a))$ 
that are obtained as images of a fixed point $f^\omega(a)$ 
of a morphism $f$ under a morphism $g$. 
In the case where the morphism $g$ is non-erasing, 
a series of papers has led Durand to a generalization of this theorem of Cobham
 \cite{Durand:1998a,Durand:1998c,Durand:2002,Durand:2011}. 
The precise definition of $\lambda$-pure morphic is too long to be discussed in this introduction. 
(It is given in Definition~\ref{def:pure}.)
But the main point of that definition is about the growth rate 
of the entries of the powers of a matrix associated with a morphism. 
Note that within the classical setting of the theorem of Cobham from 1969, 
$k$-uniform morphisms generate $k$-pure morphic words, 
$k$ being an integer greater than or equal to $2$.

\begin{theorem}[Cobham--Durand]\label{durand}
	Let $\lambda,\mu>1$ be two multiplicatively independent real numbers, i.e., 
	$\log \lambda / \log \mu\in \mathbb{R}\setminus \mathbb{Q}$.
	Let $\bu$  be a $\lambda$-pure morphic word and $\bv$ be a $\mu$-pure morphic word. 
	Let $\phi$ and $\psi$ be two non-erasing morphisms. 
	If $\bw=\phi(\bu)=\psi(\bv)$, then $\bw$ is ultimately periodic.
\end{theorem}

Let us now recall the result at the heart of our discussion in this paper. 
The following well known result in combinatorics on words 
is again attributed to Cobham. The aim is to get rid of the effacement in the two morphisms involved in the definition of an infinite morphic word.

\begin{theorem}\label{cob68}
         	Let $f$ be a morphism prolongable on a letter $a$ and $g$ a morphism 
	such that $g(f^{\omega}(a))$ is an infinite word. 
	Then there exist a non-erasing morphism $\sigma$ prolongable on  a letter $b$ 
	and a coding $\tau$ such that $g(f^{\omega}(a)) = \tau(\sigma^\omega(b))$.
\end{theorem}

Many authors have considered the problem of
getting rid of erasing morphisms when dealing with morphic words
\cite{Cobham:68,Pansiot:83,Allouche&Shallit:2003,Honkala:2009}. 
Motivations to cast a new light on this theorem are as follows.
\begin{itemize}
  \item This result is useful in the study of combinatorial properties of infinite words 
	because non-erasing morphisms are easier to deal with. 
	For instance, if $\sigma$ is non-erasing, then the sequence of lengths 
	of the words $\sigma^n(a)$ is non-decreasing.
  \item As mentioned earlier, in the study of abstract numeration systems, 
	the usual constructions lead to erasing morphisms 
	and again it would be convenient to work with non-erasing morphisms.
  \item With the notation of Theorem~\ref{durand},
	if $\phi$ is an erasing morphism and $\bu$ is a $\lambda$-pure morphic word, 
	then even though the morphic word $\phi(\bu)$ can
	be obtained as $\tau(\sigma^\omega(b))$ with a non-erasing morphism~$\sigma$ 
	and a coding~$\tau$ (thanks to Theorem~\ref{cob68}) and, 
	contrary to what was stated in \cite[Prop.~14]{Durand&Rigo:2009}, 
	the infinite word $\sigma^\omega(b)$ need not be $\lambda$-pure morphic. 
	As a counter-example to \cite[Prop.~14]{Durand&Rigo:2009},  
	take $f\colon a\mapsto abc, b\mapsto bac, c\mapsto ccc$,
	$\phi\colon a\mapsto a,b\mapsto b, c\mapsto \varepsilon$ and $\sigma\colon a\mapsto ab, b\mapsto ba$. 
	Even though $\bu=f^\omega(a)$ is $3$-pure morphic, its image
	$\phi(\bu)=\phi(f^\omega(a))=\sigma^\omega(a)$ under $\phi$ 
	is the Thue-Morse word  which is $2$-pure morphic.
	Thus, to be able to relate the growth orders of abstract numeration systems 
	and the corresponding morphisms or, as a first step towards a generalization of Cobham-Durand theorem, 
	whenever a morphic word $\bw$ can be obtained both as $g(f^\omega(a))$ and $\tau(\sigma^\omega(b))$ 
	where $\tau$ is a coding and $\sigma$ is a non-erasing morphism, 
	it is of great importance to have an in-depth analysis of the relations existing between the morphisms $f,g$ 
	and $\sigma,\tau$.
\item A first study of the admissible growth rates of recognizable sets of integers 
	within an abstract numeration system was considered in \cite{Charlier&Rampersad:2011}.
\item Cobham-Durand theorem does not apply to the case of morphisms with Perron eigenvalue equal to $1$, that is of polynomial growth. Those morphisms were only partially covered in \cite{Durand&Rigo:2009}. 
Let us point out that morphisms of polynomial growth are also studied in \cite{Mauduit:1986}. 
\item Another motivation comes from the classification of infinite words using transduction. 
	Roughly speaking, a transducer is a finite-state machine, i.e., a deterministic finite automaton 
	where transitions are labeled with input letters and (possibly empty) output words, 
	used to replace an infinite word by another one, where the $n$th output 
	depends on the first $n$ symbols of the original word \cite{sprunger}. 
	For instance, Dekking proved that morphic words are closed under transduction \cite{dekking:1994}.
\end{itemize}

In this paper, we gather all the necessary tools to deal with these erasing morphisms. 
If $f\colon A^*\to A^*$ is a morphism, one usually considers the matrix $\mathsf{Mat}_f$ 
where the entry $(\mathsf{Mat}_f)_{b,a}$ is the number of occurrences 
of the symbol $b\in A$ in the image $f(a)$, $a\in A$. 
Thus the sum of the entries of the column $a$ is the length of $f(a)$. 
In particular, it is easy to see that $((\mathsf{Mat}_f)^n)_{b,a}$ is the number of occurrences 
of the symbol $b\in A$ in the image $f^n(a)$, i.e., $((\mathsf{Mat}_f)^n)_{b,a}=(\mathsf{Mat}_{f^n})_{b,a}$. 
Thus we will keep track of the matrices associated with morphisms and study the asymptotic behavior of their powers. 
With the notation of Theorem~\ref{cob68}, our task is to relate the properties 
of the matrix $\mathsf{Mat}_f$ associated with $f$ to the matrix $\mathsf{Mat}_\sigma$ associated with $\sigma$. 

In Section~\ref{sec:2}, we first recall some classical results in linear algebra. 
We assume that the reader is more familiar with combinatorics on words than with applications of 
Perron-Frobenius theory. 
So this section is written to be self-contained. 
Our presentation avoids the use of analytic results about rational series~\cite{Salomaa-Soittola:1978} 
and should be accessible to readers having a background either in graph theory or linear algebra. 
We make use of the Perron theorem (that is plainly stated)
and we discuss properties of non-negative matrices. 
With Lemma~\ref{lemma:upper-triangular-primitive}, 
Proposition~\ref{prop:lambda-d-ijr}
and Proposition~\ref{prop:somme}, 
we carefully study the asymptotic behavior of their powers where a periodicity naturally appears. 
 We also introduce the notion of a dilated matrix and show that a non-negative matrix 
and any of its dilated versions have the same spectral radius. 
Dilatation of matrices naturally appears in the algorithm derived from Theorem~\ref{cob68}.

Section~\ref{sec:3} contains the main discussion about erasing morphisms. 
First we recall how to get rid of these morphisms. 
With the notations of Theorem~\ref{cob68}, one can effectively get the morphisms $\sigma$ and $\tau$ from $f$ and $g$. 
Then our aim is to relate the growth rate of the new non-erasing morphisms with that of the former erasing morphisms. 

Along the paper we explicitly present an algorithm derived from Theorem~\ref{cob68} in four parts (Algorithms 1 to 4). 
Thus the implementation of it can be easily realized. 

\section{Asymptotics and operations on matrices}\label{sec:2}

Matrices are naturally associated with morphisms. 
In this section, we recall some results about non-negative matrices. 

We also introduce dilatation of matrices.  
The notion provides structural information on the transformations we apply to morphisms. 
However it is not crucial for the results we obtain later on. 
It provides some extra information about Proposition~\ref{prop:algo}. 
Also it naturally appears in constructions where the product of two automata is considered 
(e.g., in the proof that any recognizable set within an abstract numeration system 
has a morphic characteristic sequence \cite{Maes&Rigo:02,rigo1}).

\subsection{Perron--Frobenius theory}
\begin{definition}
	Let $M$ be a square matrix. 
	The {\em spectrum of $M$} is the multiset of its eigenvalues (repeated with respect to their algebraic multiplicities). 
	It is denoted by $\multspec(M)$.
	The {\em spectral radius of $M$} is the real number 
	\[
	\rho(M) = \max\{|\lambda| \mid \lambda \in \multspec(M)\}.
	\]
\end{definition}

We are concerned with non-negative matrices only. 
In this section, we recall that the spectral radius of a non-negative matrix is an eigenvalue of this matrix. 
We then recall asymptotic results about powers of non-negative matrices. 

\begin{theorem}\cite{Gantmacher:59}\label{the:pf}
	If $M$ is a non-negative square matrix, then $\rho(M)$ is an eigenvalue of $M$.
\end{theorem}

In the literature, in view of Theorem~\ref{the:pf}, we also find 
 the term {\em Perron (or Perron--Frobenius) eigenvalue of $M$} 
to designate the spectral radius $\rho(M)$.

\begin{definition}
	A non-negative square matrix $M$ of size $m$ is said to be {\em primitive} 
	if there exists a positive integer $k$ such that, for all $i,j \in \{1,\dots,m\}$, $(M^k)_{i,j}>0$. 
\end{definition}

For references on the Perron theorem, see \cite{Gantmacher:59,Seneta:1981, Marcus&Lind:1995, Meyer:2000}.

\begin{theorem}[Perron theorem for primitive matrices]\label{the:p-f} 
	Let $M$ be a primitive matrix of size $m$. 
	\begin{itemize}
	  \item[(i)] The spectral radius $\rho(M)$ is positive and is an eigenvalue of $M$ which is algebraically simple.
 	\item[(ii)] Every eigenvalue $\alpha\in\mathbb{C}$ of $M$ 
			such that $\alpha\neq \rho(M)$ satisfies $|\alpha|<\rho(M)$.
	  \item[(iii)] For all $i,j\in\{1,\ldots,m\}$, there exists $c_{i,j}>0$ 
			such that $(M^n)_{i,j}/\rho(M)^n$ converges to $c_{i,j}$ as $n$ tends to infinity.
	\end{itemize}
\end{theorem}

The following result is classical in the theory of non-negative square matrices.  
For example, see \cite[Section~4.5]{Marcus&Lind:1995} for details.

\begin{prop}\label{prop:power}
	Let $M$ be a non-negative square matrix. 
	Then there exists a permutation matrix $T$ and a positive integer $p$ such that
	\begin{equation}\label{upper-block-triangular-primitive}
	T^{-1}M^pT
	\end{equation}
	is an upper (or lower) block triangular matrix where each square block on the diagonal is either primitive or~$(0)$.
	Furthermore, the least integer $p$ satisfying this condition is computable.
\end{prop}

The notation $(0)$ stands for the matrix $0_{1\times 1}$ of size $1$.

\begin{definition}\label{def:power}
	Let $M$ be an non-negative square matrix. 
	We let $\p(M)$ denote the least integer $p$ for which 
	there exists a permutation matrix $T$ such that~\eqref{upper-block-triangular-primitive} 
	is an upper (or lower) block triangular matrix where each block on the diagonal is either primitive or~$(0)$. 
\end{definition}

\begin{remark}\label{rem:digraph}
	Any matrix $M\in\N^{m\times m}$ can be interpreted as the {\em adjacency matrix} 
	of a digraph with $m$ vertices. 
	The entry $M_{i,j}$ counts the number of edges from vertex $i$ to vertex $j$. 
	In particular, it is an elementary result in graph theory that $(M^n)_{i,j}$ 
	is the number of walks of length $n$ from vertex $i$ to vertex $j$.  
	The zero blocks $(0)$ on the diagonal of~\eqref{upper-block-triangular-primitive} correspond to 
	single vertices with no loop on them. 
	Finally, the permutation $T$ in \eqref{upper-block-triangular-primitive} 
	simply corresponds to a reordering of the vertices of the graph. 
\end{remark}

The following algorithm computes the value $\p(M)$. The correctness of this algorithm follows from \cite[Chapter 4]{Marcus&Lind:1995} or \cite[Section 2.5]{rigo1}.

\begin{algo}
The input is a non-negative square matrix $M$. The output is the integer $\p(M)$.

\begin{enumerate}[(i)]
	\item Compute the digraph $G(M)$ associated with $M$.
	\item If $G(M)$ is a forest, then $\p(M)$ is $1$.
	\item Else for each non-trivial strongly connected component $I$ of $G(M)$, 
		compute the $\gcd$ of the lengths of the simple cycles in $I$, which we denote by $p_I$. 
	 Then $\p(M)$ is the $\lcm$ of the $p_I$'s.
\end{enumerate}
\end{algo}

The following lemma will be used as a recurrent argument in the proofs of this section.

\begin{lemma} \label{lemma:recurrent-argument}
	Let $M\in\N^{m\times m}$. There exists $N\in\N$ such that for all $i,j\in \{1,\dots,m\}$, 
	all $r\in\{0,\ldots,\p(M)-1\}$ red and all integers $N'\ge \lceil (m+1-r)/\p(M)\rceil$, 
	if $(M^{\p(M)N'+r})_{i,j}>0$ 
	then, for all  integers $n\ge N+N'$, $(M^{\p(M)n+r})_{i,j}>0$.
\end{lemma}

\begin{proof}
We use the graph interpretation of non-negative square matrices. 
Since the graph corresponding to $M$ has $m$ vertices,
for all $i,j\in\{1,\ldots,m\}$, every walk of length at least $m+1$ from vertex $i$ to vertex $j$ 
goes through a vertex $k$ that belongs to a cycle. 
In particular, if the permutation $T$ in \eqref{upper-block-triangular-primitive} maps $k$ to $\ell$, 
then $(T^{-1}M^{\p(M)}T)_{\ell,\ell}$ is an element of a primitive block.
 
Let $N\in\N$ be such that $P^N>0$ for all primitive blocks $P$ 
on the diagonal in \eqref{upper-block-triangular-primitive}. 
In particular, $P^n>0$ for all $n\ge N$. 
This means that $(T^{-1}M^{\p(M)n}T)_{\ell,\ell}=(M^{\p(M)n})_{k,k}>0$ for all $n\ge N$.

Let $i,j\in \{1,\dots,m\}$, $r\in\{0,\ldots,\p(M)-1\}$ and $N'\ge \lceil (m+1-r)/\p(M)\rceil$ 
be such that $(M^{\p(M)N'+r})_{i,j}>0$. 
This means that there exists a walk of length $\p(M)N'+r  \ge m+1$ 
from the vertex $i$ to the vertex $j$ in the graph corresponding to $M$. 
From the above discussion, there exist $u,v\in\N$ and a vertex $k$ 
such that $u+v=\p(M)N'+r$, $(M^u)_{i,k}>0$, $(M^v)_{k,j}>0$ and $(M^{\p(M)n})_{k,k}$ for all $n\ge N$.
Then, for all $n\ge N+N'$, 
\[
	(M^{\p(M)n+r})_{i,j}\ge (M^u)_{i,k}\, (M^{\p(M)(n-N')})_{k,k}\, (M^v)_{k,j} >0.
\]
\end{proof}

The next lemma essentially  follows from  Theorem~\ref{the:p-f}~(iii) 
and from Lemma~\ref{lemma:fraction} below 
which is a particular case of a theorem of Darboux (see for instance~\cite[Theorem 2.2]{Pemantle&Wislon}).

\begin{lemma} \label{lemma:upper-triangular-primitive}
	Let $M\in\N^{m\times m}$ be an upper block triangular matrix of the form
	\begin{eqnarray}\label{upper-triangular-primitive}
		M = 
		\left(
		\begin{array}{cccc}
		P_1	& B_{1,2}  & \cdots	& B_{1,h}	\\
		0		& P_2	& \ddots	& \vdots 	\\
		\vdots 	& \ddots	& \ddots 	& B_{h-1,h} \\
		0		& \cdots 	& 0		& P_h	
		\end{array}
		\right)	
	\end{eqnarray}
	where the diagonal square blocks $P_\ell$ are either primitive or $(0)$. 
	For all $i,j\in\{1,\ldots,m\}$, either $(M^n)_{i,j} = 0$ for all sufficiently large $n$, or  
	there exist $\lambda\in \multspec(M)\cap \R_{\ge 1}$ and $d\in \N$ 
	such that $(M^n)_{i,j} = \Theta(n^d\, \lambda^n)$.

	More precisely, in the second case, if $M_{i,j}$ is an entry of the block $B_{k,\ell}$ with $1\le k< \ell\le h$  
	or an entry of $P_\ell$ (in which case we set $k=\ell$ in the formulas), 
	then  we have
	\begin{align}
		\label{eq:lambda-primitive}
		\lambda & =\max_{\substack{k=m_1< m_2< \cdots < m_t=\ell\\ 
		B_{m_1,m_2}\neq 0,\  B_{m_2,m_3} \neq 0,\  \ldots,\  B_{m_{t-1},m_t} \neq 0}}
		\max\{\rho(P_{m_s})\mid s\in\{1,\ldots,t\}\} ; \\
	\label{eq:d-primitive}
	d+1 & =\max_{\substack{k=m_1< m_2< \cdots< m_{t-1} < m_t=\ell\\ 
	B_{m_1,m_2}\neq 0,\  B_{m_2,m_3} \neq 0,\  \ldots,\  B_{m_{t-1},m_t} \neq 0}}
	\#\{s\in\{1,\ldots,t\}\mid \rho(P_{m_s})=\lambda \}.
	\end{align}
	In particular, the asymptotic behaviors of $(M^n)_{i,j}$ corresponding to entries of a given block $B_{k,\ell}$ coincide.
\end{lemma}

\begin{lemma}
\label{lemma:fraction}
Let $(a_n)_{n\ge0}\in\R^\N$. 
Suppose that its generating function is rational:
\[\sum_{n\ge0}a_nx^n =\frac{P}{Q}\]
where $P,Q \in \R[x]$ are coprime.
Assume further that $Q$ has a single root $\alpha$ of minimal modulus 
and that this root is non-zero and has multiplicity $m$.
Then there exists $c\in\R\setminus\{0\}$ such that $a_n/(n^{m-1}\alpha^{-n})$ 
converges to $c$ as $n$ tends to infinity.
\end{lemma}

\begin{proof}
Let $x_1,\dots,x_p$ be the distinct roots of $Q$ with respective multiplicities $m_1,\dots,m_p$.
W.l.o.g we assume that $x_1=\alpha$  and $m_1=m$. 
By Euclidean division and then decomposing into partial fractions, 
there exist $C\in\R[x]$ and $c_{i,j}\in\mathbb{C}$, with $1 \leq i \leq p$ and $1 \leq j \leq m_i$, such that
\[
	\frac{P}{Q} =C+  \sum_{i=1}^p \sum_{j=1}^{m_i} \frac{c_{i,j}}{(1-x/x_i)^j}.
\]
Moreover, for each $i$ corresponding to a real root $x_i$ of $Q$, we know that $c_{i,m_i}$ is a nonzero real number. 
In particular, this is the case for $c_{1,m}$ since $\alpha\in\R$.

As $1/(1- x)^{t+1}=\sum_{n \geq 0} \binom{n+t}{t} x^n$ if $t\in\N$, we obtain that, for $n$ large enough,
\[
	a_n =  \sum_{i=1}^p \sum_{j=1}^{m_i} c_{i,j} \tbinom{n+j-1}{j-1} \frac{1}{x_i^n}.
\]
Hence the result.
\end{proof}

\begin{proof}[Proof of Lemma~\ref{lemma:upper-triangular-primitive}]
From Theorem~\ref{the:p-f}~(iii), 
we know that $P_\ell^n=C_{\ell}\, (\rho(P_\ell)^n+o(\rho(P_\ell)^n))$
where $C_\ell$ is a positive matrix.
Therefore, for all $i,j$ such that $M_{i,j}$ belongs to a block $P_\ell$ on the diagonal, 
$(M^n)_{i,j}=\Theta(\rho(P_\ell)^n)$.
Note that $\rho(P_\ell)\in\multspec(M)$ and if $P_\ell\neq (0)$ then $\rho(P_\ell)\ge 1$.
Also note that, in this case, we have found $\lambda=\rho(P_\ell)$ and $d=0$, 
which is coherent with~\eqref{eq:lambda-primitive} and~\eqref{eq:d-primitive} if we set $k=\ell$.

The blocks above the diagonal are obtained as sums of products 
involving the diagonal blocks $P_1,\ldots,P_h$ and the blocks above the diagonal. 

Consider the example where $M$ is the upper block triangular matrix given by 
\[
	M= \begin{pmatrix} 
		    A&D&F\\ 0&B&E\\ 0&0&C\\
		\end{pmatrix}
\]
where $A,B,C$ are primitive matrices or $(0)$.
We associate a labeled graph, called a {\em path graph} \cite{Lindqvist:1989}, with such a block matrix. 
Its set of vertices is the set of blocks on the diagonal. 
The block in $i$th row and $j$th column is the label of the edge from vertex $i$ to vertex $j$. 
The path graph of $M$ is depicted in Figure~\ref{fig:path}.
 \begin{figure}[htbp]
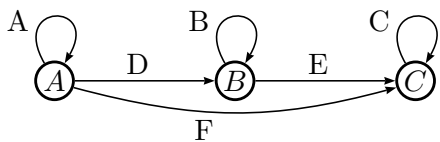

        \centering
\VCDraw{
        \begin{VCPicture}{(0,-1.3)(8,1.8)}
\State[C]{(8,0)}{c} 
\State[B]{(4,0)}{b}
\State[A]{(0,0)}{a}

\ArcR{a}{c}{$F$}
\LoopN{c}{$C$}
\LoopN{b}{$B$}
\LoopN{a}{$A$}
\EdgeL{a}{b}{$D$}
\EdgeL{b}{c}{$E$}
\end{VCPicture}
}
        \caption{The path graph associated with $M$}.
        \label{fig:path}
 \end{figure} 

As in Remark \ref{rem:digraph}, let $G$ be the directed graph whose adjacency matrix is $M$. 
Fix a vertex $a$ (resp. $c$) in the primitive component corresponding to $A$ (resp. $C$) 
or if $A$ (resp. $C$) is (0), then $a$ (resp. $c$) is the single vertex of the corresponding component. 
There are two types of walks of length $n$ from $a$ to $c$. 

\begin{itemize}
\item First, there are those that can be decomposed as a walk of length $i$ from $a$ to a vertex $a'$ in $A$,
	followed by an edge from $a'$ to a vertex $c'$ in $C$ 
	and ending with a walk of length $n-i-1$ from $c'$ to $c$. 
	The number of these walks is given by $(A^iFC^{n-i-1})_{a,c}$. 
\item Second, there are those that can be decomposed as a walk of length $i$ 
	from $a$ to a vertex $a'$ in $A$,
	followed by an edge from $a'$ to a vertex $b$ in $B$, 
	then a walk of length $j$ from $b$ to a vertex $b'$ in $B$, 
	followed by an edge from $b'$ to a vertex $c'$ in $C$,
	and ending with a walk of length $n-i-j-2$ from $c'$ to $c$.
	The number of these walks is given by $(A^iDB^jEC^{n-i-j-2})_{a,c}$.
\end{itemize}

The total number of walks of length $n$ from any $a\in A$ to any $c\in C$
is given by the entry $(M^n)_{a,c}$, 
which belongs to the block corresponding to $F$ in $M^n$. 
From the definition of the matrix product, the block corresponding to $F$ (resp., $D$, $E$) in $M^n$ 
is the sum of the labels of the walks of length $n$ from $A$ to $C$ (resp., $A$ to $B$, $B$ to $C$)
in the path graph depicted in Figure~\ref{fig:path}. 
Indeed, in our example, the upper-right block in $M^n$ is 
\[
	\sum_{\substack{i+j+k=n-2\\ i,j,k\ge 0}} A^iD B^j E C^k
	+\sum_{\substack{i+j=n-1\\i,j\ge 0}} A^iFC^j.
\]

If $\rho(A)=\rho(B)=\rho(C)$, $D\neq0$ and $E\neq0$,
since $\#\{(i,j,k)\in\N^3\mid i+j+k=n\}=\binom{n+2}{2}$, 
the entries of this block have a behavior in $\Theta(n^2\,\rho(A)^n)$. 
Note that, since $A, B$ and $C$ are primitive or $(0)$, 
there exist $i,j,k\in\N$ such that all entries of $A^iD B^j E C^k$ are simultaneously positive 
if and only if $D\neq0$ and $E\neq0$. 

If  $\beta=\rho(A)=\rho(B)>\rho(C)=\gamma$, $D\neq0$ and $E\neq0$,
then the entries of this block have a behavior in $\Theta(n\,\rho(A)^n)$ because
\[
	\sum_{\substack{i+j+k=n\\ i,j,k\ge 0}} \beta^{i+j} \gamma^k
	=\gamma^{n}\sum_{k=0}^n  (n-k+1) \left(\frac{\beta}{\gamma}\right)^{n-k}
	=\gamma^{n}\sum_{k=0}^n  (k+1) \left(\frac{\beta}{\gamma}\right)^{k}
\]
and the conclusion follows from \eqref{eq:gen_case} with $m=1$.

Let us turn to the general case. 
Recall that for $m\in\mathbb{N}$ and $\lambda\in\R_{>1}$,  we have  
\begin{equation}
    \label{eq:gen_case}
    \sum_{i=0}^n i^m  \lambda^i = \Theta(n^m \lambda^n).
\end{equation}
Indeed, we have
\[
	\int_0^n x^m \lambda^x dx \leq \sum_{i=0}^n i^m  \lambda^i \leq \int_0^{n+1} x^m \lambda^x dx
\]
and the result follows by using integration by parts and an induction on $m$. 
Note that the exact expansion of $\sum_{i=0}^n i^m  \lambda^i$ can be explicitly given 
(see for example~\cite{foata:2010}).
Also, it is a classical result of enumerative combinatorics \cite{feller:1950} that
\[
	\#\{(i_1,\ldots,i_q)\in\N^q\mid i_1+\cdots+i_q=n\}=\tbinom{n+q-1}{q-1}=\Theta(n^{q-1}).
\]

The block corresponding to $B_{k,\ell}$, $k<\ell$, in $M^n$ is a sum of terms of the form
\[
	\sum_{\substack{i_1+\cdots+i_t=n-t+1\\ i_1,\ldots,i_t\ge 0}}P_{m_1}^{i_1}B_{m_1,m_2}P_{m_2}^{i_2}
			\cdots B_{m_{t-1},m_t}P_t^{i_t}.
\]
In the above expression, we count the number of walks of length $n$ 
starting from a vertex in $P_{m_1}$, ending in a vertex in $P_{m_t}$ 
and going through the components $P_{m_2},\ldots,P_{m_{t-1}}$. 
For $n$ large enough, such a walk exists if and only if the $B_{m_i,m_{i+1}}$'s are all non-zero. 
If we consider the spectral radii of the blocks $P_{m_1},\ldots,P_{m_t}$ 
and assuming that $q$ of them are maximal, to derive the asymptotic behavior of an element, 
we have to estimate sums of the following form.

If $\beta_1=\cdots=\beta_q>\beta_{q+1}\ge \cdots\ge \beta_t\ge 1$ where $q\in\{1,\ldots,t-1\}$, then
\begin{equation}    \label{eq:asympt}
	\sum_{\substack{i_1+\cdots+i_t
	=n\\ i_1,\ldots,i_t\ge 0}} \prod_{j=1}^t \beta_j^{i_j}
	=\sum_{i=0}^n \tbinom{i+q-1}{q-1}\beta_1^i 
		\sum_{\substack{i_{q+1}+\cdots+i_t=n-i\\ i_{q+1},\ldots,i_t\ge 0}} 
		\prod_{j=q+1}^t \beta_j^{i_j}.
\end{equation}
We get
\[
	\beta_t^{n} \sum_{i=0}^n \tbinom{i+q-1}{q-1}  \left(\frac{\beta_1}{\beta_{t}}\right)^{i} 
	\le \eqref{eq:asympt} 
	\le \beta_{q+1}^{n} \sum_{i=0}^n  \tbinom{i+q-1}{q-1}\tbinom{n-i+t-q-1}{t-q-1} \left(\frac{\beta_1}{\beta_{q+1}}\right)^{i}
\]
and both the left hand side and the right hand side are in $\Theta(n^{q-1} \beta_1^n)$.
For the right hand side, this is a consequence of Lemma~\ref{lemma:fraction}. 
Indeed, the generating function of the sequence
\[
	\left(\sum_{i=0}^n \binom{i+q-1}{q-1}\binom{n-i+t-q-1}{t-q-1} \beta^i\right)_{n\ge 0}
\] 
is the Cauchy product
\[
	\left( \sum_{i \geq 0} \tbinom{i+q-1}{q-1} (\beta x)^i \right) 
	\left( \sum_{j \geq 0}\tbinom{j+t-q-1}{t-q-1} x^j \right)
			=\frac{1}{(1-\beta x)^q(1-x)^{t-q}}.
\]

If $q=t$, that is if $\beta_1=\cdots= \beta_t\ge 1$, then again
\begin{equation*}  
	\sum_{\substack{i_1+\cdots+i_t=n\\ i_1,\ldots,i_t\ge 0}} \prod_{j=1}^t \beta_j^{i_j}
	= \tbinom{n+q-1}{q-1} \beta_1^n
	=\Theta(n^{q-1}\beta_1^n).
\end{equation*}

This explains, in the statement of the result, the extra $n^d$ factor that may occur 
if several diagonal blocks have the same spectral radius. 
In other words, $\lambda=\beta_1$ is the largest spectral radius
that one can encounter on a walk from the vertex $i$ to the vertex $j$ and 
$d+1=q$ counts the maximal number of blocks on the diagonal having spectral radius $\lambda$ 
that one can encounter on the considered walks.
\end{proof}

\begin{prop}\label{prop:lambda-d-ijr}
	Let $M\in\N^{m\times m}$. 
	Then, for all $i,j \in \{1,\dots,m\}$ and $r\in\{0,\ldots,\p(M)-1\}$, 
	either $(M^{\p(M)n+r})_{i,j}=0$ for all sufficiently large $n$, 
	or there exist $\lambda\in\multspec(M^{\p(M)})\cap \R_{\ge 1}$ and $d\in\N$ 
	such that $(M^{\p(M)n+r})_{i,j}=\Theta(n^d\lambda^n)$.
\end{prop}

\begin{proof}
Let  $N\in\N$ be a constant such as in Lemma~\ref{lemma:recurrent-argument}.
Let  $i,j \in \{1,\dots,m\}$ and $r\in \{0,\ldots,\p(M)-1\}$. 
Suppose that $(M^{\p(M)n+r})_{i,j}$ is not ultimately vanishing. 
Hence there is some $N'\ge \lceil(m+1-r)/\p(M)\rceil$
such that $(M^{\p(M)N'+r})_{i,j}>0$.
By Lemma~\ref{lemma:recurrent-argument}, 
\[
	(M^{\p(M)(N+N')+r})_{i,j}=\sum_{k=1}^m (M^{\p(M)(N+N')})_{i,k} (M^r)_{k,j}>0.
\] 
Hence,
there exists $k$ such that $(M^{\p(M)(N+N')})_{i,k}(M^r)_{k,j}>0$. 
In particular, by Lemma~\ref{lemma:recurrent-argument} again $(M^{\p(M)n})_{i,k}>0$ for all $n\ge 2N+N'$. 
This means that the set
\[
	K := \{k\in\{1,\ldots,m\}\,\mid\, (M^{\p(M)n})_{i,k} (M^r)_{k,j}>0 \text{ for all } n \text{ large enough}\}
\]
is nonempty. 
Note that if $k\in\{1,\ldots,m\}\setminus K$, then $(M^{\p(M)n})_{i,k} (M^r)_{k,j}=0$ 
for all $n\ge  \lceil(m+1)/\p(M)\rceil$.
From Proposition~\ref{prop:power}, $M^{\p(M)}$ 
has the form~\eqref{upper-triangular-primitive} up to a permutation.
Then, for each $k\in K$, 
we know from Lemma~\ref{lemma:upper-triangular-primitive} that 
there exist $\lambda_k\in\multspec(M^{\p(M)})\cap \R_{\ge 1}$ and $d_k \in \N$ such that
$(M^{\p(M)n})_{i,k} =\Theta(n^{d_k}\lambda_k^n)$.
Define
\begin{align*}
	\lambda & :=\max \{\lambda_k\, \mid\,  k\in K \};\\
	d & :=\max\{d_k\,\mid\, k\in K \text{ and } \lambda_k=\lambda\}.
\end{align*}
Let $k_0\in K$ such that $\lambda_{k_0}=\lambda$ and $d_{k_0}=d$.
Then,   for all sufficiently large $n$,
\[
	(M^{\p(M)n})_{i,k_0} (M^r)_{k_0,j} \le (M^{\p(M)n+r})_{i,j}= \sum_{k\in K} (M^{\p(M)n})_{i,k} (M^r)_{k,j}.
\]
where the last equality follows from Lemma~\ref{lemma:recurrent-argument} again.
We have obtained that $(M^{\p(M)n+r})_{i,j}=\Theta(n^d\lambda^n)$, hence the result.
\end{proof}

\begin{definition}
\label{def:lambda(i,j,r)}
For every matrix $M\in\N^{m\times m}$, indices $i,j \in \{1,\dots,m\}$ and remainder $r\in\{0,\ldots,\p(M)-1\}$, 
we let $\lambda(i,j,r)$ and $d(i,j,r)$ 
denote the two quantities $\lambda$ and $d$ obtained in Proposition~\ref{prop:lambda-d-ijr}, 
if $(M^{\p(M)n+r})_{i,j}$ is not ultimately zero (as $n$ tends to infinity);  
and we set $\lambda(i,j,r)=0$ and $d(i,j,r)=0$, otherwise.
\end{definition}

The following example shows that we cannot hope for more than the previous statement 
in the sense that $\lambda$ and $d$ really depend on $i,j,r$.

\begin{example} Consider the graph depicted in Figure \ref{fig:ex} and its adjacency matrix $M$.
 \begin{figure}[htbp]
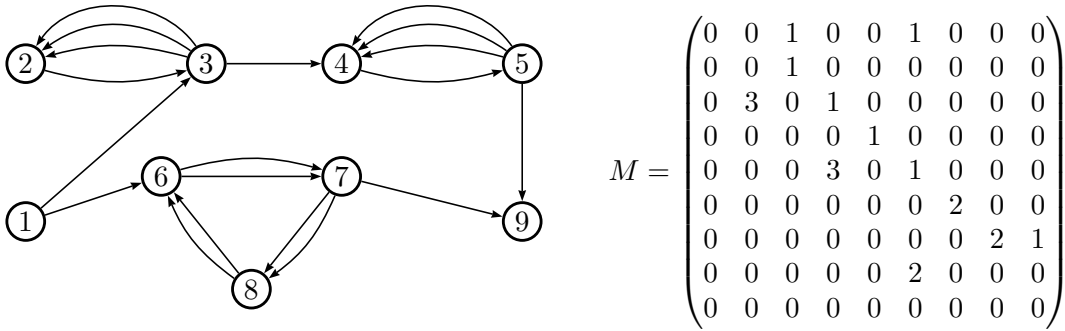

        \centering
\begin{minipage}{0.5\linewidth}
\centering
\VCDraw{
        \begin{VCPicture}{(1,-6.3)(12,1.8)}
\State[1]{(1,-3.5)}{1}
\State[2]{(1,0)}{2}
\State[3]{(5,0)}{3}
\State[4]{(8,0)}{4}
\State[5]{(12,0)}{5}
\State[9]{(12,-3.5)}{6}
\State[6]{(4,-2.5)}{7}
\State[7]{(8,-2.5)}{8}
\State[8]{(6,-5)}{9}
\EdgeR{1}{3}{}
\EdgeR{1}{7}{}
\ArcR{2}{3}{}
\ArcR{3}{2}{}
\VArcR{arcangle=-35}{3}{2}{}
\VArcR{arcangle=-55}{3}{2}{}
\EdgeL{3}{4}{}
\ArcR{4}{5}{}
\ArcR{5}{4}{}
\VArcR{arcangle=-35}{5}{4}{}
\VArcR{arcangle=-55}{5}{4}{}
\EdgeL{5}{6}{}
\EdgeL{7}{8}{}
\ArcL{7}{8}{}
\EdgeL{8}{9}{}
\ArcL{8}{9}{}
\EdgeL{9}{7}{}
\EdgeR{8}{6}{}
\ArcL{9}{7}{}
\end{VCPicture}
}
\end{minipage}
\begin{minipage}{.45\linewidth}
    $$
    M=\begin{pmatrix}
        0&0&1&0&0&1&0&0&0\\
        0&0&1&0&0&0&0&0&0\\
        0&3&0&1&0&0&0&0&0\\
        0&0&0&0&1&0&0&0&0\\
        0&0&0&3&0&1&0&0&0\\
        0&0&0&0&0&0&2&0&0\\
        0&0&0&0&0&0&0&2&1\\
        0&0&0&0&0&2&0&0&0\\
        0&0&0&0&0&0&0&0&0\\
    \end{pmatrix}$$
\end{minipage}
        \caption{A directed graph.}
        \label{fig:ex}
 \end{figure}
Note that for the square blocks corresponding to the strongly connected components 
$\{2,3\}, \{4,5\}, \{6,7,8\}$ of the graph, we have
\[
	\begin{pmatrix}
	    0&1\\
	    3&0\\
	\end{pmatrix}^2=
	\begin{pmatrix}
	    3&0\\
	    0&3\\
	\end{pmatrix},\quad 
	\begin{pmatrix}
	    0&2&0\\
	    0&0&2\\
	    2&0&0\\
	\end{pmatrix}^2=
	\begin{pmatrix}
	    0&0&4\\ 
	    4&0&0\\ 
	    0&4&0\\
	\end{pmatrix}\text{ and }
	\begin{pmatrix}
	    0&2&0\\
	    0&0&2\\
	    2&0&0\\
	\end{pmatrix}^3=
	\begin{pmatrix}
	    8&0&0\\
	    0&8&0\\ 
	    0&0&8\\
	\end{pmatrix}
\]
and thus $\p(M)=6$. 
The matrix $M^6$ is of the form \eqref{upper-block-triangular-primitive} 
with seven primitive diagonal blocks of size $1$ which are $(3^3)$ four times and $(8^2)$ 
three times, and also two blocks $(0)$ (corresponding to the vertices $1$ and $9$). 
In particular, the spectral radii of the non-trivial strongly connected components
are $3^{3/6}=\sqrt{3}$ and $8^{2/6}=2$. 
In Table \ref{tab:va}, we have represented the asymptotic behaviors 
of $(M^{\p(M)n+r})_{i,j}$ for all $r\in\{0,\ldots,5\}$ and some selected pairs $(i,j)$.
\begin{table}
\[
	\begin{array}{c|ccccccc}
	    r&(1,2) & (1,3)&(1,5)&(1,7)&(1,8)&(1,9)&(4,9)\\
	\hline
	    0& \sqrt{3}^n 	& 0 			& 0 			& 0 		& 2^n 	& 2^n 		& \sqrt{3}^n\\
	    1& 0 		  	& \sqrt{3}^n 	& n \sqrt{3}^n 	& 0 		& 0 		& 0 			& 0\\
	    2& \sqrt{3}^n 	& 0 			& 0 			& 2^n 	& 0		 & n \sqrt{3}^n 	& \sqrt{3}^n\\
	    3& 0 			& \sqrt{3}^n 	& n \sqrt{3}^n & 0 		& 2^n 	& 2^n 		& 0\\ 
	    4& \sqrt{3}^n 	& 0 			& 0 			& 0 		& 0 		& n \sqrt{3}^n 	& \sqrt{3}^n\\
	    5& 0 			& \sqrt{3}^n 	& n \sqrt{3}^n & 2^n 	& 0 		& 0 			& 0\\
	\end{array}
\]
\caption{Values of $(M^{\p(M)n+r})_{i,j}=\Theta(n^d\lambda^n)$ for some $(i,j)$ and $r\in\{0,\ldots,\p(M)-1\}$.}
    \label{tab:va}
\end{table}

Indeed, there are only walks of even (resp. odd) length from vertex $1$ to $2$ (resp. $3$). 
There are only walks of odd length from vertex $1$ to $5$ 
and the extra factor $n$ comes from the fact that those walks 
may visit the vertices $2,3,4,5$. For walks from $1$ to $9$, 
one has to take into account the walks of even length going 
through the vertices $2,3,4$ but also walks of length multiple of $3$ 
going through $6,7,8$. Since $2>\sqrt{3}$ 
the behavior for the number of walks of length multiple of $6$ is given by those going through $6,7,8$.
\end{example} 

\nopagebreak
\begin{prop}\label{prop:somme}
	Let $M\in\N^{m\times m}$.
	\begin{itemize}
	\item[(i)] For all $i\in\{1,\ldots,m\}$,  there exist $\lambda\in\multspec(M)$ and $d\in\N$ such that
			\[
				\sum_{j=1}^m(M^n)_{i,j}=\Theta(n^d\lambda^n).
			\]
	\item[(ii)] For all $j\in\{1,\ldots,m\}$, there exist $\lambda\in\multspec(M)$ and $d\in\N$ such that
			\[ 
				\sum_{i=1}^m(M^n)_{i,j}=\Theta(n^d\lambda^n).
			\]	
	\end{itemize}
\end{prop}

\begin{proof}
Let us prove (ii) (the item (i) is symmetric).
For each $j \in \{1,\ldots,m\}$ and each $r\in\{0,\ldots,\p(M)-1\}$, we define 
\begin{align*}
	\lambda(*,j,r) & :=	\max\{\lambda(i,j,r)\,\mid\, 1\le i\le m\};\\
	d(*,j,r) & :=\max\{d(i,j,r)\,\mid\, 1\le i\le m \text{ and } \lambda(i,j,r)=\lambda(*,j,r)\}.
\end{align*}
We know from Proposition~\ref{prop:lambda-d-ijr} that
\[
	\sum_{i=1}^m(M^{\p(M)n+r})_{i,j}=\Theta\bigl(n^{d(*,j,r)}\lambda(*,j,r)^n\bigr).
\]
Hence, to prove the lemma, it suffices to show that the quantities $\lambda(*,j,r)$ and $d(*,j,r)$ only depend on $j$, 
that is, for all $r,r'\in\{0,\ldots,\p(M)-1\}$, one has $\lambda(*,j,r)=\lambda(*,j,r')$ and $d(*,j,r)=d(*,j,r')$.

Let $j\in\{1,\ldots,m\}$ and $r,r'\in\{0,\ldots,\p(M)-1\}$ such that $r\neq r'$.
If $\lambda(*,j,r)=0$ then $\lambda(*,j,r')=0$. 
This is because, for all $i'\in\{1,\ldots,m\}$ and all large enough $n$, one has
\[
	(M^{\p(M)n+r'})_{i',j}=\sum_{i=1}^m (M^{\p(M)+r'-r})_{i',i} \underbrace{(M^{\p(M)(n-1)+r})_{i,j}}_{=\,  0 \text{ for each $i$}} =0.
\]
Moreover, in this case, we have $d(*,j,r)=d(*,j,r')=0$ by definition. 

Assume now that $\lambda(*,j,r)>0$. 
Let $i\in\{1,\ldots,m\}$ such that $\lambda(*,j,r)=\lambda(i,j,r)$ and $d(*,j,r)=d(i,j,r)$.
Then let $i'\in\{1,\ldots,m\}$ such that $\lambda(i',j,r')=\lambda(i,j,r)$ and $d(i',j,r')=d(i,j,r)$.
Such an index $i'$ exists because
\[
	(M^{\p(M)n+r})_{i,j}=\sum_{k=1}^m (M^{\p(M)+r-r'})_{i,k} (M^{\p(M)(n-1)+r'})_{k,j}.
\]
This implies $\lambda(*,j,r)\le\lambda(*,j,r')$. 
By symmetry  $\lambda(*,j,r)=\lambda(*,j,r')$. 
Then, since $d(*,j,r)=d(i',j,r')$, we obtain $d(*,j,r)\le d(*,j,r')$. 
Again by symmetry $d(*,j,r)= d(*,j,r')$.

To end the proof of (ii), we take $\lambda=\lambda(*,j,r)^{1/\p(M)}$ 
and $d=d(*,j,r)$ for any $r\in\{0,\ldots,\p(M)-1\}$. 
Indeed, by Euclidean division, every $n$ can be written $\p(M)\lfloor n/\p(M)\rfloor+r$ with $r<\p(M)$. 
This explains the apparition of the $\p(M)$th root.
\end{proof}

\begin{remark} \label{rem:interpretation-lambda} 
It is convenient for what follows to give a description of the quantity $\lambda$ 
that appears in the previous result.
Consider the case (ii) and fix $j\in\{1,\ldots,m\}$.
Then, what can be extracted from the proof of Proposition \ref{prop:somme} is that $\lambda$ is the greatest 
$\lambda(i,j,0)^{1/\p(M)}$ for $1 \leq i \leq m$ 
(since we have proved the independence of the $\lambda(i,j,r)$ with respect to $r\in\{0,\ldots,\p(M)\}$).
Assume now that $M^{\p(M)}$ is of the form~\eqref{upper-triangular-primitive} 
(this is always the case up to a permutation). 
In particular, from Definition~\ref{def:lambda(i,j,r)} and Lemma~\ref{lemma:upper-triangular-primitive},
we deduce that $\lambda(i,j,0)$ is the greatest spectral radius of the diagonal blocks $P_\ell$ 
for which there exist $k \in \{1,\dots,m\}$ and $n_1,n_2\in\N$ such that 
$(M^{\p(M)n_1})_{i,k}>0$,  $(M^{\p(M)n_2})_{k,j}>0$ and $(M^{\p(M)})_{k,k}$
is an entry of $P_\ell$.
Then $\lambda$ is 
 the $\p(M)$th root of the greatest spectral radius of the diagonal blocks $P_\ell$ 
for which there exists $k \in \{1,\dots,m\}$ such that $(M^{\p(M)})_{k,k}$ is an entry of $P_\ell$ and 
$(M^{\p(M)n})_{k,j}>0$ for some $n\in\N$.
\end{remark}

\subsection{Dilatation of matrices}
Roughly speaking, when dilating a matrix $M$, each entry $M_{i,j}$ is replaced 
in a convenient way by a matrix of size $k_i\times k_j$ whose lines all sum up to $M_{i,j}$.

\begin{definition}\label{def:dil}
	Let $M$ be a real square matrix of size $m$. 
	A real square matrix $D$ of size $n \geq m$ is called a {\em dilated matrix of $M$} if
	there exist positive integers $k_1 , \dots, k_m$ such that 
	\begin{itemize}
		\item[(i)] $\sum_{i=1}^m k_i = n$;
		\item[(ii)] rows and columns are both indexed by pairs $(i,k)$ for $1 \leq i \leq m$ and $1 \leq k \leq k_i$;
		\item[(iii)] $D$ satisfies the following property:
	\begin{equation} \label{eq: property row sum}
		\forall i,j \in \{1,\dots,m\}, \, \forall k \in \{1,\dots,k_i\}, \qquad \sum_{\ell = 1}^{k_j} D_{(i,k),(j,\ell)} = M_{i,j}.
	\end{equation}
	\end{itemize}
	The vector $(k_1,k_2,\dots,k_m)$ is called the {\em dilatation vector of $D$}.
	We let $\dil(M)$ denote the set of dilated matrices of $M$.
\end{definition}

In other words, given a square matrix $M$ of size $m$, a dilated matrix with dilatation vector $(k_1,\dots,k_m)$ 
of $M$ is a block matrix 
\[
	D = 
	\left(
	\begin{array}{ccc}
		B_{1,1} & \cdots & B_{1,m}	\\
		\vdots	& \ddots	 & \vdots	\\
		B_{m,1}	& \cdots & B_{m,m} 
	\end{array}	 
	\right)
\]
where each block $B_{i,j}$ has $k_i$ rows and $k_j$ columns and such that for all $k \in \{1,\dots,k_i\}$, one has 
\[
	\sum_{\ell = 1}^{k_j} (B_{i,j})_{k,\ell} = M_{i,j}.
\]

Definition~\ref{def:dil} can be adapted to column vectors instead of matrices. 
The idea is to repeat several times a given entry to be compatible and coherent
with the multiplication of a matrix with a column vector.

\begin{definition}\label{def:dilv}
	Let $x \in \R^m$ be a column vector. 
	A vector $d \in \R^n$ with $n \geq m$ is a {\em dilated vector of $x$} 
	if there exist positive integers $k_1, \dots, k_m$ such that 
	\begin{enumerate}
	\item $\sum_{i=1}^m k_i=n$;
	\item entries of $d$ are indexed by pairs $(i,k)$ for $1 \le i \le m$ and $1\le k\le k_i$;
	\item for all $i \in \{1,\dots,m\}$ and all $k \in \{1,\dots,k_i\}$, we have $d_{(i,k)} = x_i$.
	\end{enumerate}
\end{definition}

\begin{example}
Consider the following matrix $M$ and vector $x$
\[
	M=\left(
		\begin{array}{ccc}
		1 	&	1	&	1	\\
		2	&	1	&	1	\\
		1	&	1	&	0	
		\end{array}
	\right)
	\  \text{and} \ 
	x=\left(
		\begin{array}{c}
		1	\\ 0 	\\ 	2
		\end{array}
	\right).
\]
The matrix $D$ and the vector $d$ below are respectively, a dilated matrix of $M$ 
and a dilated vector of $x$ with dilatation vector $(1,2,2)$.
\[
	D=\left(
		\begin{array}{c|cc|cc}
		1 	&	1	&	0	&	0	&	1	\\
		\hline
		2	&	0	&	1	&	1	&	0	\\
		2	&	1	&	0	&	1/2	&	1/2	\\
		\hline
		1	&	\sqrt{2}	&	1-\sqrt{2}	&	1	&	-1	\\	
		1	&	0	&	1	&	0	&	0	
		\end{array}
	\right)
	\  \text{and} \ 
	d=\left(
		\begin{array}{c}
		1	\\ \hline 0 \\0		\\	\hline 2	\\ 	2
		\end{array}
	\right).
\]
Observe that the product $Dd$ is a dilated vector of the product $Mx$:
\[
	MX = (3 \ 4 \ 1)^T 
	\quad \text{and} \quad
	Dd = (3 \mid 4 \ 4 \mid 1 \ 1 )^T	
\] 
\end{example}

\begin{lemma} \label{lemma: spectrum dilated}
	Let $M$ be a real square matrix of size $m$. 
	Let $D$ be a dilated matrix of $M$.
	Each eigenvalue of $M$ is an eigenvalue of $D$.
\end{lemma}

\begin{proof}  
Assume that $D$ is a dilated matrix of $M$ with dilatation vector $(k_1,\dots,k_m)$. 
Let $\lambda$ be an eigenvalue of $M$ and let $x$ be an eigenvector of $M$ such that $M x = \lambda x$.
Let $y$ be a dilated vector of $x$ with dilatation vector $(k_1,\dots,k_m)$. 
The vector $y$ is non-zero and for all $i$ in $\{1,\dots,m\}$ and all $k \in \{1,\dots,k_i\}$, we have	
\[
	(D y)_{(i,k)} 
	=
	\sum_{j=1}^m \sum_{\ell=1}^{k_j} D_{(i,k),(j,\ell)} y_{(j,\ell)}
	=
	\sum_{j=1}^m \left(\sum_{\ell=1}^{k_j} D_{(i,k),(j,\ell)}\right) x_j
	=
	\sum_{j=1}^m M_{i,j} x_j =
	\lambda\, x_i =
	\lambda\,  y_{(i,k)}.
\]
Hence, $\lambda$ is an eigenvalue of $D$.
\end{proof}

\begin{prop}\label{pro:dil}
	Let $M$ be a non-negative square matrix. For any non-negative matrix $D$ in $\dil(M)$,  
	$M$ and $D$ have the same spectral radius.
\end{prop}

\begin{proof} 
We follow the lines of the proof of \cite[Proposition 7]{Nicolay&Rigo:07}.
Due to Lemma~\ref{lemma: spectrum dilated} and Theorem~\ref{the:pf}, 
we have $\rho(D) \geq \rho(M)$. 
Let us prove that we also have $\rho(D) \leq \rho(M)$.

The Collatz-Wielandt formula (see, for instance, \cite[Chap.~8]{Meyer:2000}) states that, 
for any primitive (or even, for any irreducible) matrix $N$ of size $m$, 
\[
	\rho(N) = \max_{\substack{y \in \mathbb{R}^m \\ y \geq 0}} 
			\min_{\substack{1 \leq i \leq m\\ y_i \neq 0}} \frac{(N y)_i}{y_i}.
\]

Let $m$ (resp., $n$) be the size of $M$ (resp., $D$).  
Let us first suppose that $M$ and $D$ are primitive.
Let us prove that for all non-negative vectors $y \in \mathbb{R}^n$ 
there is a non-negative vector $x \in \mathbb{R}^m$ such that
\[
	\min_{\substack{1 \leq i \leq n \\ y_i \neq 0}} \frac{(Dy)_i}{y_i}
	\leq 
	\min_{\substack{1 \leq i \leq m \\ x_i \neq 0}} \frac{(Mx)_i}{x_i}.
\] 
Let $y$ be a non-negative vector in $\mathbb{R}^n$ and let $(k_1,\dots,k_m)$ be the dilatation vector of $D$. 
With the convention taken in Definition~\ref{def:dilv}, 
we index the components of $y$ by the ordered pairs $(i,k)$ for $1 \leq i \leq m$ and $1 \leq k \leq k_i$. 
Let us define the non-negative vector $x \in \mathbb{R}^m$ by 
\[
	x_i = \max_{1 \leq k \leq k_i} y_{(i,k)}.
\]
We have 
\begin{eqnarray*}
	\min_{\substack{1 \leq i \leq m \\ 1 \leq k \leq k_i \\ y_{(i,k)} \neq 0}} \frac{(D y)_{(i,k)}}{y_{(i,k)}} 
	&=&
	\min_{\substack{1 \leq i \leq m \\ 1 \leq k \leq k_i \\ y_{(i,k)} \neq 0}} \frac{1}{y_{(i,k)}} \sum_{j=1}^m \sum_{\ell=1}^{k_j}	
		D_{(i,k),(j,\ell)} y_{(j,\ell)}		\\
	&\leq &  \min_{\substack{1 \leq i \leq m \\ 1 \leq k \leq k_i \\ y_{(i,k)} \neq 0}} \frac{1}{y_{(i,k)}} \sum_{j=1}^m 
		\left( \sum_{\ell=1}^{k_j} 	D_{(i,k),(j,\ell)} \right) x_j		 \\
	&\leq &	\min_{\substack{1 \leq i \leq m \\ 1 \leq k \leq k_i \\ y_{(i,k)} \neq 0}} \frac{1}{y_{(i,k)}} \sum_{j=1}^m M_{i,j} x_j\\
	&=&\min_{\substack{1 \leq i \leq m \\ x_i \neq 0}} \frac{1}{x_i}  \sum_{j=1}^m M_{i,j}  x_j.
\end{eqnarray*}
This completes the case for primitive matrices.

Now suppose that $M$ or $D$ is  not primitive. Let $J$ be the $n\times n$ matrix 
whose entries are all equal to $1$. Let $C$ be the $m\times m$ matrix defined by $C_{i,j}=k_j$ for all $i,j$. 
We can consider sequences of  matrices $(M_s)_{s\ge 1}$ and $(D_s)_{s\ge 1}$ 
where $M_s = M+\frac{1}{s}C$ (resp., $D_s = D +\frac{1}{s}J$). 
Note that $M_s$ and $D_s$ are positive matrices, hence primitive. 
Moreover, $J$ is a dilated matrix of $C$ with dilatation vector $(k_1,\ldots,k_m)$. 
Hence the same holds for $D_s$ and $M_s$. We can therefore apply the same reasoning 
as in the first part of the proof and obtain $\rho(D_s) \leq \rho(M_s)$ for all $s\ge 1$. 
Since $\lim_{s \to +\infty} \rho(M_s) = \rho(M)$ and $\lim_{s \to +\infty} \rho(D_s) = \rho(D)$, 
we conclude that $\rho(D) \leq \rho(M)$. 
\end{proof}


\section{Growth orders of morphisms used to generate morphic words}\label{sec:3}

In the first part of this section, we recall classical definitions on infinite words that can be obtained as the image under a morphism $g$ 
of the infinite word generated by  iteratively applying another prolongable morphism $f$ on an initial letter $a$.  
It is well known that such a word $g(f^\omega(a))$ can also be obtained with a coding $\tau$ 
and a non-erasing morphism $\sigma$, i.e., $g(f^\omega(a))=\tau(\sigma^\omega(b))$. 
We discuss this result in the second part of this section, and relate precisely the growth rates of $f$ and $\sigma$. 

\subsection{Basic definitions}\label{sec:basic}

Let $A$ be an alphabet. The set of finite words over $A$ is denoted by $A^*$. 
Endowed with the concatenation product, this set is a monoid whose neutral element is the empty word $\varepsilon$. 
We set $A^+=A^*\setminus\{\varepsilon\}$. 
The length of a word $w\in A^*$ is denoted by $|w|$ and the number of occurrences of the letter $a$ in $w$ 
is denoted by $|w|_a$. We have $|\varepsilon|=0$. 
A morphism $f\colon A^*\to B^*$ is a {\em coding} if, for all $a\in A$, $|f(a)|=1$. 
It is said to be {\em non-erasing} if, for all $a\in A$, $|f(a)|\ge 1$. 
Moreover, morphisms defined over $A^*$ can naturally be extended over $A^\mathbb{N}$. 
For more, see \cite{cant:2010,rigo1}.

\begin{definition}
	Let $A$ be an alphabet and $f\colon A^* \to A^*$ be a morphism. 
	We call a letter $a \in A$ {\em mortal (w.r.t.~$f$)} if there is a positive integer $n$ such that $f^n(a)=\varepsilon$. 
	A non-mortal letter is called {\em immortal (w.r.t.~$f$)}. 
	We let $A_{\mathcal{M},f}$ (or simply $A_\mathcal{M}$) denote the set of mortal letters 
	and $A_{\mathcal{I},f}$ (or simply $A_\mathcal{I}$) the set of immortal letters.
\end{definition}

\begin{definition}
	A subset $B$ of an alphabet  $A$ is said to be a {\em sub-alphabet} of $A$.
	In this case, we let $\kappa_{A,B}\colon A^* \to (A\setminus B)^*$ 
	denote the morphism defined by $\kappa_{A,B}(a)=\varepsilon$ if $a \in B$ 
	and $\kappa_{\mathcal{M},B}(a)=a$ otherwise.
\end{definition}

\begin{definition}
	Let $f \colon A^* \to A^*$ be a morphism. 
	The {\em incidence matrix} of $f$ is the matrix $\mathsf{Mat}_f\in\N^{A\times A}$ defined, for all $a,b\in A$, by
	\[
		{(\mathsf{Mat}_f)}_{a,b} = |f(b)|_a.
	\]
	For every sub-alphabet $B$ of $A$, we let 
	\[
		(\mathsf{Mat}_f)_B
	\] 
	denote the sub-matrix of $\mathsf{Mat}_f$ obtained from $\mathsf{Mat}_f$ 
	by selecting rows and columns corresponding to letters in $B$. 
	The eigenvalues and the spectrum of $\mathsf{Mat}_f$ are called respectively the {\em eigenvalues} 
	and the {\em spectrum} of $f$ which is denoted by $\multspec(f)$. Since $\mathsf{Mat}_f$ is non-negative, thanks to Theorem~\ref{the:pf} 
we can also define the {\em Perron eigenvalue} of $f$, which is $\rho(\mathsf{Mat}_f)$.
\end{definition}

\begin{definition}\label{def:restr}
	Let $f \colon A^* \to A^*$ be a morphism and let $B \subseteq A$ be a sub-alphabet.
	If $f(B)\subseteq B^*$, we say that the restriction 
	$f_B := \restriction{f}{B^*}\colon B^* \to B^*$ is a {\em sub-morphism} of $f$.
\end{definition}

Observe that for all $f$, $f_{A_\mathcal{M}}$ is a sub-morphism of $f$. 
Also, if $f_B$ is a sub-morphism of $f$, then $(\mathsf{Mat}_{f_B})=(\mathsf{Mat}_f)_B$. 

\begin{remark}
For any morphism $f \colon A^* \to A^*$ and for all $n \in \N$, 
we have $\mathsf{Mat}_{f^n} = \mathsf{Mat}_f^n$. 
Let $A=\{a_1,\ldots,a_n\}$, we denote by $\Psi(w)$ the column vector 
$(|w|_{a_1},\ldots,|w|_{a_n})^T$ for every finite word $w\in A^*$. 
The incidence matrix of a morphism $f$ satisfies
\[
	\mathsf{Mat}_f\Psi(w)=\Psi(f(w)).
\]
\end{remark}

\begin{remark}\label{rem:submorphism-matrix}
	Let $f \colon A^* \to A^*$ be a morphism such that $\mathsf{Mat}_f$ is of the form~\eqref{upper-triangular-primitive} 
	(or, equivalently, a lower block triangular matrix with primitive or $(0)$ blocks on the diagonal), 
	and let $B\subseteq A$ be a sub-alphabet such that $f(B)\subseteq B^*$, i.e., $f_B$ is a sub-morphism. Then $\mathsf{Mat}_{f_B}$ is  not any sub-matrix of $\mathsf{Mat}_f$. It is a block sub-matrix of $\mathsf{Mat}_f$ composed of entire blocks $B_{k,\ell}$ from the original block decomposition~\eqref{upper-triangular-primitive} 
where we set $B_{k,k}=P_k$ and $B_{k,\ell}=0$ if $k>\ell$
	(mutatis mutandis, if we consider the lower block triangular equivalent form).
	This is because, for all diagonal blocks $P_\ell$, we have 
	\[
		\{a\in A\,\mid\, (\mathsf{Mat}_f)_{a,a} \text{ belongs to } P_\ell\}\cap B \neq \emptyset \implies 
		\{a\in A\,\mid\, (\mathsf{Mat}_f)_{a,a} \text{ belongs to } P_\ell\}\subseteq B.
	\]
	Indeed, suppose the converse. So there is a block $P_\ell$ and letters $a\in A\setminus B$ and $b\in B$ such that 
	$(\mathsf{Mat}_f)_{a,a}$ and $(\mathsf{Mat}_f)_{b,b}$ belong to $P_\ell$.
	In this case, $P_\ell$ is not $(0)$, hence it is primitive.
	But then there exists a positive integer $n$ such that $(\mathsf{Mat}_f^n)_{a,b}= |f^n(b)|_a>0$, 
	which contradicts the hypothesis that $f(B)\subseteq B^*$.
	
	In particular, $\mathsf{Mat}_{f_B}$ is of the same block triangular form as  $\mathsf{Mat}_f$ 
	where the square blocks on the diagonal are some of the  primitive or $(0)$ blocks on the diagonal of $\mathsf{Mat}_f$.
\end{remark}

The next result is a reformulation of Proposition~\ref{prop:somme} (ii) in terms of morphisms (see also~\cite{Cassaigne-Mauduit-Nicolas:arxiv} or~\cite[Chap. 4]{cant:2010}). 

\begin{prop}
	Let $f\colon A^* \to A^*$ be a morphism.
	For all $a \in A$, there exist $d\in \N$ and $\lambda \in \multspec(f)$ 
	such that $|f^n(a)|= \Theta(n^{d}\, \lambda^n)$.
\end{prop}

\begin{proof}
    Simply observe that $$|f^n(a)|=\sum_{b\in A}|f^n(a)|_b=\sum_{b\in A} (\mathsf{Mat}_f^n)_{b,a}.$$
\end{proof}

\begin{definition}
    	The unique $d\in\N$ and $\lambda\ge 0$ associated with $a\in A$ in the above lemma 
	are denoted by $d(f,a)$ and $\lambda(f,a)$ respectively 
	(or simply, $\lambda(a)$ and $d(a)$ if there is no ambiguity on $f$).
\end{definition}

\begin{definition}\label{def:prolongable}
	A morphism $f \colon A^* \to A^*$ is {\em prolongable} on a letter $a\in A$ if $f(a)=au$ 
	for some $u\in A^+$ and $\lim_{n\to +\infty}|f^n(a)|=+\infty$. 
\end{definition}

\begin{remark}
	If a morphism $f$ is prolongable on a letter $a$, then the letter $a$ is not mortal and either $\lambda(a) >1$, or $\lambda(a) = 1$ and $d(a)\geq 1$.
\end{remark}
\begin{definition}
	An infinite word $\bw$ over $A$ is said to be {\em pure morphic} 
	if there is a morphism $f \colon A^* \to A^*$ prolongable on the first letter $a$ of $\bw$
	such that $\bw = f^\omega(a) := \lim_{n \to +\infty} f^n(a)$. 
	Convergence of a sequence of finite words to an infinite word is classical; 
	see, for instance, \cite{cant:2010}. 
	An infinite word is {\em morphic} if it is a morphic image of a pure morphic word. 
\end{definition}

Note that in the definition of a morphic word, the second morphism need not be a coding.

\begin{prop}	\label{prop:lambda-a}
	Let $f \colon A^* \to A^*$ be a morphism prolongable on the letter $a$. 
	If all letters of $A$ occur in $f^\omega(a)$, then $\lambda(a)=\rho(\mathsf{Mat}_f)$. 
\end{prop}

\begin{proof}
	Let $p=\p(\mathsf{Mat}_f)$. 
	Without loss of generality, we can suppose that $(\mathsf{Mat}_f)^p$ 
	is of the from~\eqref{upper-triangular-primitive}. 
	From Remark~\ref{rem:interpretation-lambda} 
	we know that $\lambda(a)$ is the $p$th root of the greatest spectral radius
	of the diagonal blocks $P_\ell$ in~\eqref{upper-triangular-primitive} 
	for which there exists $b\in A$ such that $((\mathsf{Mat}_f)^p)_{b,b}$ is an entry of $P_\ell$ and 
	$((\mathsf{Mat}_f)^{pn})_{b,a}=|f^n(a)|_b>0$ for some $n\in\N$.
	As  all letters of $A$ occur in $f^\omega(a)$, for every $b\in A$, we have $|f^n(a)|_b>0$ for some $n\in\N$.
	Hence the conclusion: $\lambda(a)=\max\{\rho(P_\ell)^{1/p}\,\mid\, 1\le \ell\le h\}=\rho(\mathsf{Mat}_f)$. 
\end{proof}

\begin{definition}\label{def:pure}
	An infinite word $\bw$ over $A$ is said to be {\em $(\lambda,d)$-pure morphic} if 
	\begin{itemize}
	\item there exists a morphism $f \colon A^* \to A^*$ prolongable on the first letter $a$ of $\bw$ 
		such that $\bw = f^\omega(a)$;
	\item $\lambda=\lambda(f,a)$ and $d=d(f,a)$;
	\item all letters of $A$ occur in $\bw$.
	\end{itemize}	
	The pair $(\lambda,d)$ is called the {\em growth type} of $f$ w.r.t.~$a$. 
	If $\lambda$ is greater than $1$, the morphism $f$ is said to be {\em exponential} w.r.t.~$a$.
	In this case, we usually omit the information on the
	degree $d$, and simply mention that we have a $\lambda$-pure morphic word. 
	Otherwise, if $\lambda=1$, the morphism $f$ is said to be {\em polynomial of degree $d$} w.r.t.~$a$.
\end{definition}

\begin{remark}\label{rem:pure}
	As in~\cite{Durand:2011}, we impose in the definition of a pure morphic word 
	that all letters of the alphabet of the morphism occur in $\bw$. 
	This is required to have well-defined $(\lambda,d)$-pure morphic words. 
	Indeed, consider the morphism $f\colon \{0,1,2\}^* \to \{0,1,2\}^*$ defined by $f(0)=0001$, $f(1)=12$ and $f(2)=21$. 
	The  Perron eigenvalue of $f$ is $3$, but we do not want to say that $f^\omega(1)$ is $3$-pure morphic. 
	With the definition we consider, the restriction $f'$ of $f$ to $\{1,2\}^*$ provides the $2$-pure morphic word $f'^\omega(1)$.
\end{remark}

\subsection{Avoiding erasing morphisms}\label{sec:era}

The following result is classical. 

\begin{theorem}\cite{Cobham:68}  \label{thm: cobham erasing}
	Let $\bw$ be a morphic word.
	Then there exist a non-erasing morphism $\sigma$ prolongable on  a letter $b$ 
	and a coding $\tau$ such that $\bw = \tau(\sigma^\omega(b))$.
\end{theorem}

In this section, given a morphism $f$ prolongable on a letter $a$ and a morphism $g$ 
such that $g(f^{\omega}(a))$ is an infinite word, we present an algorithm 
to obtain a morphism $\sigma$ and a coding $\tau$ as in Theorem~\ref{thm: cobham erasing}. 
Our main contribution is an in-depth analysis of the respective growth types of $f$ and $\sigma$.

Proofs of Theorem~\ref{thm: cobham erasing} can be found 
in \cite{Pansiot:83,Allouche&Shallit:2003,Cassaigne&Nicolas:2003} or \cite{Honkala:2009} where the strategy is a factorization into elementary morphisms. 
Since our aim is to compare the growth type of $\sigma$ 
with that of $f$, we present a constructive proof of it, mainly based on~\cite{Cassaigne&Nicolas:2003}.
The algorithm is divided into three steps.  
First, one shows that the morphisms $f$ and $g$ can be chosen to be non-erasing.
This step is omitted in \cite{Cassaigne&Nicolas:2003}. 
The second step is a technicality that ensures that the length of the images under $(g \circ f^n)$ is non-decreasing with $n$.
The last step consists in building the morphisms $\sigma$ and $\tau$. 

\begin{lemma} \label{lemma 1}
	Let $f\colon A^* \to A^*$ be a morphism prolongable on a letter $a$.
	Let $k=\# A_\mathcal{M}$ be the number of mortal letters of $f$.
	Then the morphism
	\[
		f_\mathcal{I} := \restriction{(\kappa_{A,A_\mathcal{M}} \circ f)}{A_{\mathcal{I}}^*}
		\colon A_\mathcal{I}^* \to A_\mathcal{I}^*
	\]
	is non-erasing and such that $f^{\omega}(a)= f^k(f_{\mathcal{I}}^\omega(a))$. 
	Moreover, we have:
	\begin{itemize}
	\item For all $\ell\in\mathbb{Z}_{\ge k}$ and all $n\in\mathbb{Z}_{\ge 1}$,
		$f^\ell\circ f_\mathcal{I}^n=\restriction{f^{n+\ell}}{{A_{\mathcal{I}}^*}}$;
	\item $\mathsf{Mat}_{f_\mathcal{I}}=(\mathsf{Mat}_f)_{A_\mathcal{I}}$.
	\end{itemize}
\end{lemma}

\begin{proof}
First observe that $a\in A_\mathcal{I}$ and that $f_\mathcal{I}$ is non-erasing by definition.

Since $k$ is the number of mortal letters, it follows that  $f^k(b) = \varepsilon$ for all $b \in A_\mathcal{M}$. 
Indeed, proceed by contradiction and suppose that there exists $b \in A_\mathcal{M}$ such that $f^k(b)\neq\varepsilon$. 
Then $b,f(b),\ldots,f^k(b)$ are non-empty words over $A_\mathcal{M}$ and for each $i$, $f^{i+1}(b)$ 
must contain a letter not occurring in $b,\ldots,f^i(b)$. 
Hence the number of mortal letters would be greater than $k$.

We set $\kappa_\mathcal{M}=\kappa_{A,A_\mathcal{M}}$ and $\bw=f^{\omega}(a)$. 
Observe that $\bw=f^k(\bw)$. 
Then we also have $\bw= f^k\circ \kappa_\mathcal{M} (\bw)$. 

It remains to prove that $\kappa_\mathcal{M}(\bw)=f_{\mathcal{I}}^\omega(a)$. 
First, we show by induction on $n$ that 
\begin{equation}   \label{eq:lambda}
	(\kappa_\mathcal{M}\circ f)^n=\kappa_\mathcal{M}\circ f^n
\end{equation}
for all positive integers $n$. The result is obvious for $n=1$. 
We get
\[
	(\kappa_\mathcal{M}\circ f)^{n+1}
	=\kappa_\mathcal{M}\circ f \circ (\kappa_\mathcal{M}\circ f)^n
	= \kappa_\mathcal{M}\circ f \circ \kappa_\mathcal{M}\circ f^n
\]
where we used the induction hypothesis for the last equality. 
To conclude with the induction step, observe that 
$\kappa_\mathcal{M}\circ f \circ \kappa_\mathcal{M}=\kappa_\mathcal{M}\circ f$. 
It is a consequence of the fact that, for all $b\in A_{\mathcal{M}}$, $f(b)\in A_{\mathcal{M}}^*$.

On the one hand, $\kappa_\mathcal{M}\circ f^n(a)$ tends to $\kappa_\mathcal{M}(\bw)$ as $n\to +\infty$.  
On the other hand, thanks to \eqref{eq:lambda},  for all $n\ge 1$, 
$\kappa_\mathcal{M}\circ f^n(a)=(\kappa_\mathcal{M}\circ f)^n(a)=f_\mathcal{I}^n(a)$ 
which tends to $f_\mathcal{I}^\omega(a)$ as $n\to +\infty$. 
By uniqueness of the limit, it follows that $\kappa_\mathcal{M}(\bw)=f_{\mathcal{I}}^\omega(a)$.

We turn to the second part of the proof. 
As $f^k(b) = \varepsilon$ for all $b \in A_\mathcal{M}$, we have that, for all $\ell\ge k$, $f^\ell\circ \kappa_\mathcal{M}=f^\ell$.
Then, for all $b\in A_\mathcal{I}$, 
\begin{eqnarray*}
	f^\ell\circ f_\mathcal{I}^n(b) &=& f^\ell\circ (\kappa_{\mathcal{M}} \circ f)^n(b)\\ 
							&=& f^\ell\circ \kappa_{\mathcal{M}} \circ f^n(b)\\
							&=& f^{n+\ell}(b).
\end{eqnarray*}

To conclude with the proof, up to a permutation (corresponding to a reordering of the alphabet 
where all the immortal letters appear first), the matrix $\mathsf{Mat}_f$ can be written as 
\[
	\begin{pmatrix}
	    (\mathsf{Mat}_f)_{A_\mathcal{I}}& 0\\ \star & (\mathsf{Mat}_f)_{A_\mathcal{M}}\\
	\end{pmatrix}.
\]
Hence $\mathsf{Mat}_{f_\mathcal{I}}=(\mathsf{Mat}_f)_{A_\mathcal{I}}$. 
\end{proof}

The idea of the next statement is to remove the largest sub-morphism of $f$ 
whose alphabet is erased by $g$ (the result is stated in a slightly more general form 
where we consider any sub-morphism whose alphabet is erased by $g$). 
Recall that the notation $f_C$, for a sub-morphism of $f$, was introduced in Definition~\ref{def:restr}.

\begin{lemma} \label{lemma 2}
	Let $f\colon B^* \to B^*$ be a morphism prolongable on a letter $a$
	and $g \colon B^* \to A^*$ be a morphism such that 
	$g(f^\omega(a))$ is an infinite word. 
	Let $C$ be a sub-alphabet of $\{b\in B\mid g(b)=\varepsilon\}$ such that 
	$f_C$ is a sub-morphism of $f$.
	Then the morphisms 
	\[
		f_\varepsilon:= \restriction{(\kappa_{B,C} \circ f)}{(B \setminus C)^*}\colon (B\setminus C)^*\to (B\setminus C)^*  
		\ \text{ and }\ 
		g_\varepsilon := \restriction{g}{(B \setminus C)^*}\colon (B\setminus C)^*\to A^*
	\]
	are such that $g(f^\omega(a)) = g_\varepsilon(f_\varepsilon^\omega(a))$.
	Moreover, we have:
	\begin{itemize}
	\item For all $n\in\N$, $g_\varepsilon \circ f_\varepsilon^n = \restriction{(g \circ f^n)}{(B \setminus C)^*}$;
	\item $\{b\in B\,\mid\, g(f^n(b)) \neq \varepsilon \ \text{for all large enough } n\}\subseteq B \setminus C$;
	\item $\mathsf{Mat}_{f_\varepsilon}=(\mathsf{Mat}_f)_{B \setminus C}$.
	\end{itemize}
\end{lemma}

\begin{proof}
Let us prove that $ g(f^\omega(a)) = g_\varepsilon(f_\varepsilon^\omega(a))$.
We have $g =g_\varepsilon \circ \kappa_{B,C}$. 
Since $f(C)\subseteq C^*$, 
we can use exactly the same reasoning as in~\eqref{eq:lambda} and get, for all $n\in\mathbb{Z}_{\ge 1}$, 
\begin{equation}\label{eq:lemme 2}
	\kappa_{B,C} \circ f^n =\kappa_{B,C} \circ f^n \circ \kappa_{B,C}=(\kappa_{B,C} \circ f)^n.
\end{equation}
Hence for all $n\in\mathbb{Z}_{\ge 1}$,
\begin{eqnarray*}
	 g(f^\omega(a)) 
	&=& g_\varepsilon\circ \kappa_{B,C} \circ f^n (f^\omega(a))	\\
        &=& g_\varepsilon\circ \kappa_{B,C} \circ f^n \circ \kappa_{B,C} (f^\omega(a))	\\
	&=& g_\varepsilon\circ (\kappa_{B,C} \circ f)^n \circ \kappa_{B,C}(f^\omega(a))	\\
	&=& g_\varepsilon\circ f_\varepsilon^n  \circ \kappa_{B,C}(f^\omega(a)).
\end{eqnarray*}
We have $a \notin C$ and $f_\varepsilon(a) \in a (B \setminus C)^+$, for otherwise  $g(f^\omega(a))$ would be finite.
Thus, $f_\varepsilon$ is prolongable on $a$ and 
\[
	\kappa_{B,C}(f^\omega(a))=f_\varepsilon^\omega(a).
\]

We turn to the second part of the proof. 
First, using~\eqref{eq:lemme 2}, we obtain that for all $n\in\N$ and all $b \in B \setminus C$,
\[
	 g_\varepsilon \circ f_\varepsilon^n(b)
		=g \circ (\kappa_{B,C} \circ f)^n(b) 
		= g \circ \kappa_{B,C} \circ f^n(b) 
		= g \circ f^n(b).
\]
Then, since $f(C)\subseteq C^*$ and all letters are  $C$ is erased by $g$, 
we have that for all $b\in C$ and all $n\in\N$, $g(f^n(b)) = \varepsilon$.
Therefore $\{b\in B\,\mid\, g(f^n(b)) \neq \varepsilon \ \text{for all large enough} n\}\subseteq B \setminus C$.
Finally, from the construction of $f_\varepsilon$, we have (up to a reordering of the alphabet)
\[
	\mathsf{Mat}_f = 
	\left(
	\begin{array}{cc}
		\mathsf{Mat}_{f_\varepsilon}	&	0	\\
		\star					&	\mathsf{Mat}_{f_C}
	\end{array}
	\right).
\]
\end{proof}

The next proposition concludes with the first step of the algorithm that consists in getting rid of the effacement. This leads to Algorithm \ref{algo1} given below. 
The proof goes by iterating the previous two lemmas.

\begin{prop} \label{prop:f-g-non-effacants}
	Let $f\colon B^* \to B^*$ be a morphism prolongable on a letter $a$
	and $g \colon B^* \to A^*$ be a morphism such that $g(f^\omega(a))$ is an infinite word. 
	Let $p = \mathsf{p}(\mathsf{Mat}_f)$ as in Definition~\ref{def:power}.
	and let $B'$ be the following sub-alphabet of $B$:
	\[
		B'=\{b\in B\,\mid\,  g(f^{pn}(b)) \neq \varepsilon\ \text{for all large enough } n\}.
	\]
	Then there exist non-erasing morphisms $f'\colon B'^* \to B'^*$ and $g' \colon B'^* \to A^*$ 
	such that $g(f^\omega(a)) = g'(f'^\omega(a))$.
	More precisely 
	$f'=\restriction{(\kappa_{B,B\setminus B'} \circ f^p)}{B'^*}$ 
	and $\mathsf{Mat}_{f'} = (\mathsf{Mat}_{f^p})_{B'}$.	
\end{prop}

\begin{proof}
By definition of $p$, the matrix $\mathsf{Mat}_{f^p}$ is equal (up to a permutation matrix) 
to a lower block triangular matrix whose diagonal blocks are either primitive or $(0)$.
We just need to iterate the previous two lemmas to get the morphisms $f'$ and $g'$.

First, Lemma~\ref{lemma 1} applied to $f_0:=f^p$ and $g_0:=g$ 
provides a morphism $g_1$ and a non-erasing morphism $f_1$ 
defined over a sub-alphabet $B_1 \subseteq B$ 
such that $g(f^\omega(a))  = g_1(f_1^\omega(a))$ 
and we have $\mathsf{Mat}_{f_1} = (\mathsf{Mat}_{f^p})_{B_1}$.
Indeed, with the notation of Lemma~\ref{lemma 1}, we have 
$g_1= g \circ f^{pk}$, 
$f_1 =  \restriction{(\kappa_{B,B_\mathcal{M}} \circ f^p)}{B_{\mathcal{I}}^*}$ 
and $B_1=B_{\mathcal{I}}$. 
Moreover, $B'\subseteq B_1$ by construction and $\restriction{g_1\circ f_1^n=g\circ f^{p(n+k)}}{B_1^*}$. 
Therefore $B'=\{b\in B_1\,\mid\, g_1(f_1^n(b)) \neq \varepsilon \ \text{for all large enough } n\}$. 
Since $(f^p)_{B\setminus B_1} = (f^p)_{B_\mathcal{M}}$ is a sub-morphism of $f^p$, 
$\mathsf{Mat}_{f_1}$ is a lower block triangular matrix 
whose diagonal blocks are some of the diagonal blocks of $\mathsf{Mat}_{f^p}$ 
(see Remark~\ref{rem:submorphism-matrix}).

We apply Lemma~\ref{lemma 2} to $f_1,g_1$ 
and the largest sub-alphabet $C$ of $B_1\cap g_1^{-1}(\varepsilon)$ 
such that $(f_1)_C$ is a sub-morphism of $f_1$. 
We obtain new morphisms $g_2$ and $f_2$ defined over a sub-alphabet $B_2\subseteq B_1$ 
such that $g(f^\omega(a)) = g_2(f_2^\omega(a))$. 
Moreover, $B'\subseteq B_2$ by construction 
and $g_2\circ f_2^n=\restriction{(g_1\circ f_1^n)}{B_2^*}=\restriction{(g\circ f^{p(n+k)})}{B_2^*}$. 
Therefore $B'=\{b\in B_2\,\mid\, g_2(f_2^n(b)) \neq \varepsilon \ \text{for all large enough } n\}$. 
Further, $\mathsf{Mat}_{f_2}= (\mathsf{Mat}_{f_1})_{B_2} = (\mathsf{Mat}_{f^p})_{B_2}$.
Again, observe that $\mathsf{Mat}_{f_2}$ is a lower block triangular matrix 
whose diagonal blocks are some of the diagonal blocks of $\mathsf{Mat}_{f^p}$.

Observe that the new morphism $f_2$ might be erasing:
This is the case when a letter $b \in B_1$ is not erased by $g_1$, but is such that $g_1(f_1(b))=\varepsilon$ 
(such a letter is called {\em moribund} in \cite[Definition 7.7.2]{Allouche&Shallit:2003}).
This is why we need to iterate the process: we  iteratively apply
Lemma~\ref{lemma 1} followed by Lemma~\ref{lemma 2} 
(applied to the largest possible sub-alphabet)
until $f_\ell=f_{\ell+2}$ for some even $\ell$.
This always happens since the two applied lemmas remove letters from a finite alphabet.

The obtained morphism $f_\ell$ is necessarily non-erasing because 
when the stabilization occurs the application of Lemma~\ref{lemma 1} provides no new morphism, 
which means that $B_\ell$ contains no mortal letter with respect to $f_\ell$.
Moreover, the application 
of Lemma~\ref{lemma 2} provides no new morphism either, 
so there is no non-empty sub-alphabet $C$ of $B_\ell\cap g_\ell^{-1}(\varepsilon)$ such that $f_\ell(C)\subseteq C^*$. 
Since $f_\ell$ is non-erasing, this implies: 
\begin{equation} \label{eq:tostrength}
  \text{For all letters }b\in B_\ell,\   g_\ell(f_\ell^n(b)) \neq \varepsilon \text{ for infinitely many }n.
\end{equation}

We now have $g(f^\omega(a))=g_\ell(f_\ell^\omega(a))$ 
where $f_\ell\colon B_\ell^* \to B_\ell^*$ is a non-erasing morphism, $g_\ell \colon B_\ell^* \to A^*$ 
 and $B'=\{b\in B_\ell\,\mid\, g_\ell(f_\ell^n(b)) \neq \varepsilon \ \text{for all large enough } n\}$. 
More precisely, we have
\begin{equation}\label{eq-prop}
	g_\ell\circ f_\ell^n=\restriction{g\circ f^{p(n+k_1+k_3+\cdot+k_{\ell-1})}}{B_\ell^*}
\end{equation}
where $k_i$ is the number of mortal letters in $B_i$ with respect to $f_i$,
and
\[
	\mathsf{Mat}_{f_\ell}=(\mathsf{Mat}_{f^p})_{B_\ell}
\]
where the diagonal square blocks of $\mathsf{Mat}_{f_\ell}$ are all primitive or $(0)$.
We take $f'=f_\ell$. 
What remains to show is that $B'=B_\ell$ and that there exists a power $f_\ell^N$ of $f_\ell$ such that 
the morphism $g':=g_\ell\circ f_\ell^N$ is non-erasing.

We claim that we can strengthen \eqref{eq:tostrength} as follows: 
\begin{equation}\label{eq:strong}
	\text{For all letters }b \in B_\ell, \text{ there exists } N_b\in\N
	\text{ such that for all } n \geq N_b,\  g_\ell(f_\ell^n(b)) \neq \varepsilon.
\end{equation}
Together with \eqref{eq-prop} this implies that $B'=B_\ell$, 
whence the choice $N = \max\{N_b \mid b \in B_\ell\}$ is suitable for the definition of $g'$.

Let us prove~\eqref{eq:strong}. 
Wet let $P_1, \dots, P_t \subseteq B_\ell$ denote the sub-alphabets such that, 
for all $i \in \{1,\dots,t\}$, $(f_\ell)_{P_i}$ is a primitive sub-morphism of $f_\ell$ 
(such sub-alphabets exist, for otherwise the word $g_\ell(f_\ell^\omega(a))$ would be finite).
For every $i \in \{1,\dots,t\}$, there is a letter $c \in P_i$ such that $g_\ell(c) \neq \varepsilon$, 
for otherwise this would contradict the definition of $\ell$.
Recall that a non-negative square matrix $M$ of size $m$ is primitive 
if and only if there is an integer $k \leq m^2 - 2m +2$ such that $M^k>0$
(see, for instance, \cite[Corollary 8.5.8]{Horn&Johnson:2013}).  
Thus there is an integer $k \leq (\#B_\ell)^2 - 2 \#B_\ell+2$ such that for every $i \in \{1,\dots,t\}$,
every letter $c \in P_i$ and all integers $n\ge k$, all letters of $P_i$ occur in $f_\ell^n(c)$.
Now, as $f_\ell$ is non-erasing, for every letter $b \in B_\ell$, 
there is a non-negative integer $n_b \leq \#B_\ell$ 
such that $f_\ell^{n_b}(b)$ contains an occurrence of a letter in $\bigcup_{1 \leq i \leq t} P_i$.
Finally we can take $N_b = n_b + k$.

To conclude with the proof, we note that $f'$ is the morphism
\begin{eqnarray*}
	f_\ell &=& \restriction{(\kappa_{B_{\ell-1},B_{\ell-1}\setminus B_\ell}
							\circ \cdots 
							\circ \kappa_{B_1,B_1\setminus B_2}
							\circ \kappa_{B,B\setminus B_1}
							\circ f^p)}
					{B_\ell^*}\\
		&=& \restriction{(\kappa_{B,B\setminus B_\ell} \circ f^p)}{B_\ell^*}\\
		&=& \restriction{(\kappa_{B,B\setminus B'} \circ f^p)}{B'^*}.
\end{eqnarray*}
\end{proof}

  \begin{remark}
      In the proof of the previous result, we apply iteratively first Lemma~\ref{lemma 1} 
	and next Lemma~\ref{lemma 2}. 
	Note that we would get exactly the same result by first applying Lemma~\ref{lemma 2} 
	and then Lemma~\ref{lemma 1}.
  \end{remark}

Let us provide an algorithm that allows to get rid of the effacement.
The correctness of this algorithm is ensured by the previous proposition.

\begin{algo}\label{algo1} 
The input is two morphisms $f\colon B^* \to B^*$ and $g\colon B^* \to A^*$ 
such that $f$ is prolongable on $a \in B$.
The output is two non-erasing morphisms $f'\colon B'^* \to B^*$ and $g'\colon B'^* \to A^*$  
defined over a sub-alphabet $B'$ of $B$ containing $a$
such that $g'(f'^\omega(a)) = g(f^\omega(a))$. 

\begin{enumerate}[(i)]
	\item Define $p = \mathsf{p}(\mathsf{Mat}_f)$ as in Definition~\ref{def:power} and replace $f$ with $f^p$.
	
	\item Define $B_\mathcal{M} = \{b \in B \mid f^{\#B}(b) = \varepsilon\}$, 
		$k = \#B_\mathcal{M}$ and $C$ as the largest subset of $B \cap (g \circ f^k)^{-1}(\varepsilon)$ 
		such that $f(C) \subseteq C^*$.
		Replace $B$ with $B \setminus C$, $f$ with $\restriction{(\kappa_{B,C} \circ f)}{(B \setminus C)^*}$ and $g$ 
		with $\restriction{(g \circ f^k)}{(B \setminus C)^*}$.

	\item Repeat (ii) until $C$ is the empty set. 
		Then set $f'=f$ and $B' = B$.

	\item Define $N \leq (\#B')^2 - \#B' +2$ as the least integer 
		such that $g(f'^N(b)) \neq \varepsilon$ for all $b \in B'$.
		Then set $g' = g \circ f'^N$. 
\end{enumerate}
\end{algo}

\begin{corollary}
	Let $f,g,A,B,a,f',g',B'$ and $p$ be as in Proposition~\ref{prop:f-g-non-effacants}.
	Then $f^p_{B\setminus B'}$ is a sub-morphism of $f^p$.
	Moreover, if $f$ has growth type $(\lambda,d)$ with respect to~$a$, 	
	then exactly one of the following situations occurs:
	\begin{enumerate}
	\item $\lambda^p \notin \spec(f^p_{B \setminus B'})$ and $f'$ has growth type $(\lambda^p,d)$ w.r.t.~$a$;
	\item $\lambda^p \in \spec(f^p_{B \setminus B'})$ and 
		$\lambda \notin \spec(f')$ 
		and there exist $\lambda' \in \spec(f)$ and $d' \in \N$ 
		such that $\lambda'<\lambda$ and $f'$ has growth type $(\lambda'^p,d')$ w.r.t.~$a$;
	\item $\lambda^p \in \spec(f^p_{B \setminus B'})$ 
		and $\lambda \in \spec(f')$
		and $f'$ has growth type $(\lambda^p,d')$ w.r.t.~$a$ for some $d' \leq d$.
	\end{enumerate}
\end{corollary}

\begin{proof}
The alphabet $B \setminus B'$ being the set of letters $b$ such that 
$g(f^{pn}(b)) = \varepsilon$ for infinitely many $n$, 
we have $f^p(B \setminus B') \subseteq (B \setminus B')^*$.
Thus $f^p_{B\setminus B'}$ is a sub-morphism of $f^p$.
Furthermore, we can suppose that $\mathsf{Mat}_{f^p}$ is of the form
\[
	\mathsf{Mat}_{f^p} = 
	\left(
	\begin{array}{cccc}
	P_1 		& 0 			& \cdots		& 0			\\
	B_{2,1} 	& P_2		& \ddots		& \vdots		\\
	\vdots 	& \ddots		& \ddots 		& 0			\\
	B_{h,1}	& \cdots 		& B_{h,h-1}	& P_h	
	\end{array}
	\right)
\]
where the diagonal square blocks $P_{\ell}$ are either primitive or $(0)$.
From Proposition~\ref{prop:f-g-non-effacants} we know that the morphism $f'$ is the morphism 
$\restriction{(\kappa_{B,B\setminus B'} \circ f^p)}{B'^*}$.
Then by Remark~\ref{rem:submorphism-matrix}, up to a reordering of the letters, we can suppose that 
\[
	\mathsf{Mat}_{f^p} = 
	\left(
	\begin{array}{cc}
		\mathsf{Mat}_{f'}		&	0	\\
		\star				&	\mathsf{Mat}_{f_{B \setminus B'}^p}
	\end{array}
	\right),
\]
and that any primitive block $P_\ell$ is a diagonal block either of $\mathsf{Mat}_{f'}$ 
or of $\mathsf{Mat}_{f_{B \setminus B'}^p}$. 
 Mutatis mutandis, the result then follows from Lemma~\ref{lemma:upper-triangular-primitive}.
\end{proof}

Now we turn to the second part of the algorithm that consists in a technicality 
that ensures that the length of the images $(g \circ f^n)(b)$ is non-decreasing with $n$.
This can be done by considering powers of the morphism $f$.
Note that this is the second time that we replace $f$ with one of its power. 
The first time was in Proposition~\ref{prop:f-g-non-effacants}.

Let us recall the following lemma whose proof can be found in~\cite{Cassaigne&Nicolas:2003}.
Another proof of this result can be found~\cite{Durand:2013} 
where it is shown that $p$ and $q$ can be algorithmically chosen.
We recall the algorithm and prove its correctness in the particular case 
we are dealing with: $f$ and $g$ are non-erasing and the incidence matrix of $f$ 
is a lower block triangular matrix whose diagonal blocks are primitive or (0).

\begin{lemma} \cite[Lemme~4]{Cassaigne&Nicolas:2003} \label{lemme4}
	Let $f\colon B^* \to B^*$ be a morphism prolongable on a letter $a$
	and $g \colon B^* \to A^*$ be a morphism such that $g(f^\omega(a))$ is an infinite word. 
	Then there exist positive integers $p$ and $q$ such that
	\[
		|(g\circ f^p)(f^q(a))| > |(g\circ f^p)(a)|  \quad \text{ and } \quad \forall b \in B,\ |(g\circ f^p)(f^q(b))| \geq |(g\circ f^p)(b)|.
	\]
\end{lemma}

\begin{algo} \label{algo 2}
The input is two non-erasing morphisms $f\colon B^* \to B^*$ and $g\colon B^* \to A^*$ 
such that $f$ is prolongable on $a$ and $\mathsf{Mat}_f$ is a lower block triangular matrix whose 
diagonal blocks are primitive or (0).
The output is two non-erasing morphisms $f'\colon B^* \to B^*$ and $g'\colon B^* \to A^*$ 
such that $g'(f'^\omega(a)) = g(f^\omega(a))$ and
\begin{equation}    \label{eq:cn}
    |g'(f'(a))| > |g'(a)|  \quad \text{ and } \quad  \forall b \in B,\ |g'(f'(b))| \geq |g'(b)|.
\end{equation}
\begin{enumerate}[(i)]
	\item For all $b \in B$ such that $f^{\#B-1}(b) = f^{\#B}(b)$, 
		define $p_b$ as the least non-negative integer $n$ such that $f^n(b) = f^{n+1}(b)$.

	\item Define $p = \max\{p_b \mid b \in B, f^{\#B-1}(b) = f^{\#B}(b)\}$ and set $g' = g \circ f^p$.

	\item For all $b \in B$ such that $f^{\#B-1}(b) \neq f^{\#B}(b)$, 
		choose $k_b,\ell_b \leq \#B$ such that
		$|f^{k_b}(b)|_c>0$, $|f^{\ell_b}(c)|_c>0$ and $|f(c)| \geq 2$ for some letter $c \in B$.

	\item Define $q = \max \{k_b + \ell_b(|g'(b)|-1) \mid b \in B, f^{\#B-2}(b) \neq f^{\#B-1}(b)\}$ and set $f' = f^q$.
\end{enumerate}
\end{algo}

\begin{proof}[Correctness of Algorithm~\ref{algo 2}]
For all $q \in\mathbb{Z}_{\ge 1}$, $a$ is a proper prefix of $f^q(a)$. 
As for all $p \in \N$, $g \circ f^p$ is non-erasing, 
the condition $|(g \circ f^p)(f^q)(a)| > |(g \circ f^p)(a)|$ is always satisfied.
Let us now concentrate on the other condition.

Using that the diagonal blocks of $\mathsf{Mat}_f$ are primitive or (0), 
we can easily show that a letter $b \in B$ is {\em non-growing}, i.e., is such that $(|f^n(b)|)_{n\in\N}$ is bounded, 
if and only if there exists a positive integer $p_b \leq \#B-1$ such that $f^{p_b}(b) = f^{p_b+1}(b)$.
Furthermore, in such a case $f^n(b) = f^{n+1}(b)$ for all integers $n \geq p_b$. 
Thus, for $p$ and $g'$ as defined in the algorithm, 
we have $g'(f^n(b)) = g'(b)$ for all non-growing letters $b$ and all $n\in\N$.

Now, if a letter $b$ is {\em growing}, i.e., is such that $(|f^n(b)|)_{n\in\N}$ 
is unbounded, then there exists $c \in B$ such that $|f(c)| \geq 2$,  $|f^{\ell_b}(c)|_c>0$ 
and $|f^{k_b}(b)|_c>0$ for some $k_b,\ell_b\leq \# B-1$.
Thus, for all $n\in\mathbb{Z}_{\ge 1}$, $|f^{k_b + n \ell_n}(b)| \geq n+1$.
Define 
\[
	q = \max\{k_b + \ell_b( |g'(b)| -1) \mid b \in B, f^{\#B-1}(b) \neq f^{\#B}(b) \}
	\quad \text{and} \quad 
	f' = f^q.
\]
The sequence $(|f^n(b)|)_{n\ge 0}$ being non-decreasing, 
we get that for every growing letter $b$, $|g'(f'(b))| \geq |g'(b)|$.
We of course still have $g'(f'(b))= g'(b)$ for every non-growing letter, 
hence the result for all letters in $B$.
\end{proof}

We finally consider the last part of the algorithm.
The correctness of this algorithm is provided by Proposition~\ref{prop:algo}. 

\begin{algo}\label{algo:cn} 
The input is two non-erasing morphisms $f \colon B^* \to B^*$ 
and $g \colon B^* \to A^*$ satisfying~\eqref{eq:cn}. 
The output is two morphisms $\sigma\colon \Pi^*\to \Pi^*$ 
and $\tau\colon\Pi^*\to A^*$ defined over a new alphabet $\Pi$ 
and a letter $b\in\Pi$ such that $\sigma$ is a non-erasing morphism prolongable on $b$, 
$\tau$ is a coding  and $\tau(\sigma^\omega(b)) = g(f^\omega(a))$. 

\begin{enumerate}[(i)]
\item
Define the alphabet 
\[
	\Pi = \{(b,i) \mid b \in B, 0\leq i < |g(b)|\},
\]
the morphism
\[
	\alpha \colon B^* \to \Pi^*,\, b\mapsto (b,0)(b,1) \cdots (b,|g(b)|-1)
\]
and the coding 
\[
	\tau \colon \Pi^* \to A^*,(b,i)\mapsto (g(b))_i.
\]

\item
For any letter $b \in B$, pick a factorization $(w_{b,i})_{0 \leq i < |g(b)|}$ such that
\begin{equation}   \label{eq:alphaf}
 	\alpha(f(b)) = w_{b,0}\, w_{b,1}\,  \cdots\,  w_{b,|g(b)|-1}   
\end{equation}
with $w_{b,i} \in \Pi^+$ for all $i$ and $|w_{a,0}| \geq 2$.

\item
Define the morphism
\[
	\sigma \colon \Pi^* \to \Pi^*,(b,i)\mapsto w_{b,i}.
\]
\end{enumerate}
\end{algo}

\begin{prop} \label{prop:algo}
	Let $f\colon B^* \to B^*$ be a non-erasing morphism prolongable on a letter $a$ of growth type $(\lambda,d)$ w.r.t a,
	and $g \colon B^* \to A^*$ be a non-erasing morphism such that $g(f^\omega(a))$ is an infinite word. 
	Suppose moreover that $f$ and $g$ satisfy~\eqref{eq:cn}. 
	Then the morphisms $\tau$ and $\sigma$ built in Algorithm~\ref{algo:cn} are such that 
	$g(f^\omega(a)) = \tau(\sigma^\omega((a,0)))$, $\sigma$ is non-erasing and $\tau$ is a coding. 
	Moreover, $\mathsf{Mat}_\sigma$ is a dilated matrix of $\mathsf{Mat}_f$ 
	and $\sigma$ has growth type $(\lambda,d)$ w.r.t. $(a,0)$.
\end{prop}

\begin{proof}
Since $g$ is non-erasing, the alphabet $\Pi$ the morphism $\alpha$ and the coding $\tau$ are well defined.
Then, the existence of the factorization $(w_{b,i})_{0 \leq i < |g(b)|}$ 
with $w_{b,i} \in \Pi^+$ for all $i$ and $|w_{a,0}| \geq 2$ is ensured by~\eqref{eq:cn}, 
which makes the morphism $\sigma$ well-defined.

It is clear that $\tau$ is a coding and that $\sigma$ is non-erasing and prolongable 
on the first letter of $w_{a,0}$ which is $(a,0)$. 
By definition $\tau \circ \alpha = g$. 
Let $\bu = f^\omega(a)$ and $\bw =g( f^\omega(a))$. 
Hence $\tau(\alpha(\bu)) = \bw$. 
Let us show that  $\sigma^\omega((a,0)) = \alpha(\bu)$. 
From~\eqref{eq:alphaf} we observe that 
\[
	\alpha \circ f = \sigma \circ \alpha
\] 
which implies that $\alpha(\bu)$ is a fixed point of $\sigma$: 
$\sigma(\alpha(\bu)) = \alpha(f(\bu)) = \alpha(\bu)$.

Let us prove that $\mathsf{Mat}_\sigma \in \dil(\mathsf{Mat}_f)$ with dilatation vector $(|g(b)|)_{b \in B}$.
For all $b_i, b_j \in B$ and for all $k \in \{0,1,\dots,|g(b_i)|-1\}$, we have
\begin{eqnarray*}
	\sum_{\ell=0}^{|g(b_j)|-1} {(\mathsf{Mat}_\sigma)}_{(b_i,k),(b_j,\ell)} 
	&=&	\sum_{\ell=0}^{|g(b_j)|-1} |\sigma((b_j,\ell))|_{(b_i,k)}	\\
	&=&	\sum_{\ell=0}^{|g(b_j)|-1} |w_{b_j,\ell}|_{(b_i,k)}  \\
	&=& |w_{b_j,0}\, w_{b_j,1}\,  \cdots\,  w_{b_j,|g(b_j)|-1}|_{(b_i,k)}\\ 
	&=& |\alpha(f(b_j))|_{(b_i,k)}\\
	&=& (\mathsf{Mat}_f)_{b_i,b_j}.
\end{eqnarray*}
Indeed, if $(\mathsf{Mat}_f)_{b_i,b_j}= x$ for some $b_i,b_j \in B$, 
then the word $(b_i,0)(b_i,1) \cdots (b_i,|g(b_i)|-1)$ 
occurs $x$ times in $\alpha(f(b_j))$. 
Therefore, we also have $|\alpha(f(b_j))|_{(b_i,k)} = x$ for all $k \in \{0,1,\dots,|g(b_i)|-1\}$.

Finally, let us prove that $\sigma$ has growth type $(\lambda,d)$ w.r.t. $(a,0)$.
As $g$ is non-erasing, we have $|g(f^n(a))| = \Theta(\lambda^n n^d)$.
Then, since $|g(b)| = |\alpha(b)|$ for all $b \in B$ and $\alpha \circ f = \sigma \circ \alpha$, 
we get that $|\sigma^n (\alpha(a))| = |\alpha (f^n(a))|=|g (f^n(a))|=\Theta(\lambda^n n^d)$.
Finally, as $\sigma^n((a,0))$ converges to $\alpha(\bu) = \sigma(\alpha(\bu))$ 
when $n$ increases, there are some integers $k_1,k_2 \in \N$ 
such that $\sigma \circ \alpha(a)$ is a prefix of $\sigma^{k_1}((a,0))$, 
itself a prefix of $\sigma^{k_2} ( \alpha(a))$.
Thus, for all $n\in\N$, we have 
$|\sigma^{n+1}(\alpha(a))| \leq |\sigma^{n+k_1}((a,0))| \leq |\sigma^{n+k_2}(\alpha(a))|$, 
meaning that $\sigma$ has growth type $(\lambda,d)$ w.r.t. $(a,0)$. 
\end{proof}

\section{Acknowledgment}
We thank Manon Stipulanti for her feedback when reading a first draft of this paper.

\bibliographystyle{alpha}
\bibliography{biblio}

\end{document}